\documentclass[11pt,twoside]{article}

\usepackage[affil-it]{authblk}
\usepackage{booktabs}
\usepackage{amsfonts,amsmath,amssymb,amsthm}
\usepackage[linesnumbered,ruled,vlined,noend]{algorithm2e}
\usepackage{booktabs}
\usepackage{float}
\usepackage{enumitem}
\usepackage{lscape}
\usepackage{tikz}
\usepackage{afterpage}
\usepackage{multirow}
\usepackage{url}
\usepackage[numbers]{natbib}
\usepackage{fancyhdr}

\newtheorem{theorem}{Theorem}
\newtheorem{lemma}{Lemma}

\newtheorem{proposition}{Proposition}
\newtheorem{remark}{Remark}
\newtheorem{definition}{Definition}

\def\network{{\cal N}}
\def\arcset{{\cal A}}
\def\nodeset{{\cal V}}
\def\pathset{{\cal P}}
\def\bfamily{{\cal F}}
\def\framework{NF-F}
\newtheorem*{rcvf_algorithm}{Variable-Fixing Algorithm}

\def\mytitle{Exact solution of network flow models with strong relaxations}

\begin{document}
\fancypagestyle{firststyle}
{
   \fancyhf{}
\topskip 30pt\headsep 0pt\headheight 0pt
\renewcommand{\headrulewidth}{0pt}
}

\fancypagestyle{mypagestyle}{%
  \fancyhf{}
  \fancyhead[OR]{V.L. de Lima, M. Iori, and F.K. Miyazawa}
  \fancyhead[OL]{\thepage}
  \fancyhead[EL]{\mytitle}
  \fancyhead[ER]{\thepage}
  \renewcommand{\headrulewidth}{.4pt}
}
\pagestyle{mypagestyle}
\date{}

\title{\Large \bf \mytitle}
\author[1]{Vin{\'i}cius L. de Lima}
\author[2]{Manuel Iori}
\author[1]{Fl{\'a}vio K. Miyazawa}
\affil[1]{Institute of Computing, University of Campinas (Brazil)}
\affil[2]{DISMI, University of Modena and Reggio Emilia (Italy)}
\date{}
\maketitle
\vspace*{-5ex}
\noindent

\begin{abstract}
We address the solution of Mixed Integer Linear Programming (MILP) models with strong relaxations that are derived from Dantzig-Wolfe decompositions and allow a pseudo-polynomial pricing algorithm. We exploit their network-flow characterization and provide a framework based on column generation, reduced-cost variable-fixing, and a highly asymmetric branching scheme that allows us to take advantage of the potential of the current MILP solvers. We apply our framework to a variety of cutting and packing problems from the literature. The efficiency of the framework is proved by extensive computational experiments, in which a significant number of open instances could be solved to proven optimality for the first time.
\end{abstract}

\noindent
{\bf Keywords:} Dantzig-Wolfe Decomposition; Network Flow; Strong Relaxation; Variable Selection; Variable-Fixing.

\thispagestyle{firststyle}
\section{Introduction} \label{sec:introduction}

Mixed Integer Linear Programming (MILP) is one of the most popular mathematical programming tools to solve optimization problems.
{\em Dantzig-Wolfe} (DW) decomposition \cite{DW61} has been extensively used to enhance the strength of MILP models in many applications, by reformulating an original {\em compact model} into a so-called {\em Dantzig-Wolfe Master} (DWM) problem where variables represent integer solutions of polyhedra induced by a set of constraints from the compact model.
A significant importance of models with stronger linear relaxation is to avoid an excessive enumeration in branch-and-bound trees. Moreover, as the polyhedron of their linear relaxation is closer to the convex hull of the integer solutions, their optimal fractional solutions may provide good hints on how to derive good-quality integer solutions.

Models derived from DW decomposition may be stronger than the original compact models, but they usually require an exponential number of variables. Consequently, their linear relaxation is typically solved by column generation (CG), a method where a restricted set of variables is initially considered and new variables are iteratively generated by solving a pricing problem (see, e.g., \cite{LD05}).
The exact solution of a DWM usually relies on sophisticated branch(-and-cut)-and-price (B\&P) algorithms, in which CG (and primal cuts) is embedded within branch-and-bound. These algorithms are state-of-the-art for many problems, but they are complex to implement, and their efficiency usually relies on problem-specific techniques.

The relation between integer solutions of bounded polyhedra with paths in acyclic networks allows us to associate a discrete pricing problem with an acyclic decision network in which each path is associated with a DWM variable (see, e.g., \cite{V00}). Hence, the DWM can be seen as a {\em path flow model} (where variables correspond to the flow on each {\em path} of the network). The interest in such representation is that it allows us to derive an {\em arc flow model} (where variables correspond to the flow on each {\em arc}) having the same linear relaxation value but exponentially fewer variables. In particular, any DWM of exponential size that allows a pseudo-polynomial pricing algorithm can be transformed into an equivalent arc flow model of pseudo-polynomial size.

To the best of our knowledge, Val\'erio de Carvalho \cite{V99} was the first to solve a pseudo-polynomial arc flow model in practice, over two decades ago, by developing a B\&P for the cutting stock problem (CSP). Still, the popularity of such models increased only in the last decade, following the consistent improvements in general MILP solvers. The fact that pseudo-polynomial arc flow models have the same strength, and a much smaller size, than their equivalent path flow models has allowed their solution by general MILP solvers to be an efficient alternative to sophisticated B\&P algorithms in many applications (see, e.g., \cite{LACIV21}).
However, large-scale instances often produce huge networks, which become a strong limitation in the efficiency of the MILP solver. In addition, in cases where the linear relaxation is not very strong, specialized B\&P algorithms are still a better alternative than the solution of an arc flow model by a MILP solver. Further research in this area is thus highly envisaged.

In this paper, we propose a general framework, which we call {\em network flow framework} (\framework), to address the issue raised by huge networks in the solution of arc flow models, while exploiting the potential of CG and general MILP solvers. We focus on arc flow models of pseudo-polynomial size with very strong linear relaxation, where: either (i) finding an optimal solution is hard, but proving optimality is easy; or (ii) finding an optimal solution is easy, but proving optimality is hard.
Both cases happen, for instance, in the classical pattern-based formulation for the CSP \cite{GG61,GG63}, which is a path flow model derived by a DW decomposition, in which the optimal integer solution value is conjectured
to be either $\lceil z_{lb} \rceil$ (case (i)) or $\lceil z_{lb} \rceil + 1$ (case (ii)), where $z_{lb}$ is the optimal dual bound (see, e.g., \cite{CDDIR15}).

\framework{} is effective in solving case (i) instances thanks to an innovative and highly asymmetric branching scheme in which a series of small-sized arc flow models are solved by a general MILP solver, in the spirit of the classical local branching \cite{FL03}. These models are expected to be sufficiently small to be quickly solved, but large enough to likely contain an optimal solution. In instances of case (ii), once an optimal solution is at hand, optimality is typically proven by completely exploring a branch-and-bound tree. This process is usually accelerated by the use of primal cuts to strengthen the relaxation. However, in many cases, the most effective cuts are problem-dependent and usually produce a heavy impact on the pricing problem, while the huge number of arcs may still be an unaddressed major issue. NF-F does not exploit the use of general primal cuts but rather focuses on the use of reduced-cost variable-fixing (RCVF) procedures, which are usually very effective for instances of case (ii). We particularly exploit the impact of sub-optimal dual solutions in variable-fixing with a two-fold interest: providing an increased reduction in the network size and possibly strengthening the linear relaxation.

We also provide specific contributions regarding the linear relaxation solution of arc flow and path flow models, as: a formalization of a dual correspondence between these models; a non-trivial CG algorithm that generates multiple paths and can be used to solve the relaxation of both models; and a generalization of a method to deal with dual-infeasibility in CG due to the limited precision of most Linear Programming (LP) solvers.
We provide applications to cutting and packing (C\&P) problems that allow pseudo-polynomial models with very strong relaxation, namely, the CSP, the two-stage guillotine CSP, the skiving stock problem, and the ordered open-end bin packing problem. For all such problems, we performed extensive computational experiments that prove the outstanding performance of \framework.

The remainder of this paper is organized as follows. Section \ref{sec:foundations} presents a brief review on network flow models derived from DW decompositions. Section \ref{sec:overview} provides an overview of \framework. Sections \ref{sec:linear_relaxation}, \ref{sec:variable_fixing}, and \ref{sec:branching} present the details of the techniques that we use to solve the linear relaxation, apply RCVF, and perform branching, respectively. 
Section \ref{sec:applications} discusses the applications to a number of C\&P problems. The outcome of the computational experiments is discussed in Section \ref{sec:experiments}, where we also contrast our results with those obtained by state-of-the-art algorithms. Finally, Section \ref{sec:conclusions} gives conclusions and future research directions.
A preliminary version of this work appeared in \cite{LIM21}.

\section{Network Flow and Dantzig-Wolfe Decompositions: Preliminaries} \label{sec:foundations}

Due to a well-known relation between integer solutions of polyhedra with paths in acyclic networks (see, e.g., \cite{V00}), a DWM can be directly represented as a path flow model. Then, based on the flow decomposition theorem by Ahuja et al. \cite{AMO93}, we can derive an arc flow model of the same strength (i.e., same linear relaxation value). This section briefly reviews path flow and arc flow models derived from DW decompositions.

A network $\network$ is a directed graph with a set of nodes $\nodeset$ and a set of arcs $\arcset \subseteq \nodeset \times \nodeset$. Two special nodes in $\nodeset$ are the source $v^+$ and the sink $v^-$, for which no arcs in $\arcset$ enters $v^+$ or leaves $v^-$. The set of all paths from $v^+ $to $v^-$ is given by $\pathset$, and the set of all arcs of a path $p \in \pathset$ is given by $\arcset_p$. The set of all paths in $\pathset$ that contain an arc $(u,v) \in \arcset$ is given by $\pathset_{(u,v)}$.

We assume that $\network$ is acyclic, as we are only concerned with decision networks underlying pricing problems from DW decompositions. In this context, $\nodeset$ and $\arcset$ represent, respectively, the states and decisions of a pricing algorithm. Each arc/decision $(u,v) \in \arcset$ is associated with a cost $c_{(u,v)} \in \mathbb{Z}$ and a contribution $a_{(u,v)} \in \mathbb{Z}^m$ to the constraints induced by the problem. Each path $p \in \pathset$ in the pricing network represents a DWM variable with cost $c_p = \sum_{(u,v) \in \arcset_p} c_{(u,v)}$ and constraint coefficients given by column $a_p = \sum_{(u,v) \in \arcset_p} a_{(u,v)}$. Then, the DWM can be formulated as a path flow model:
\begin{align}
\allowdisplaybreaks
\label{eq:path_flow_of}
 \min & \sum_{p \in \pathset}c_p \lambda_p, & \\
\label{eq:path_flow_constraints}
 \text{s.t.:} & \sum_{p\in \pathset}a_p\lambda_p \geq b, & \\
\label{eq:path_flow_domain}
 & \lambda_p \in \mathbb{Z}_+,  & \forall  p \in \pathset,
\end{align}

\noindent where $b \in \mathbb{Z}^m$ and each variable $\lambda_p \in \mathbb{Z}^+$ represents the flow on path $p \in \pathset$.
The following general arc flow model is equivalent to \eqref{eq:path_flow_of}--\eqref{eq:path_flow_domain}:
\begin{align}
\allowdisplaybreaks
\label{eq:arc_flow_of}
\min ~ & \sum_{(u,v) \in {\arcset}} c_{(u,v)}\varphi_{(u,v)}, & \\
\label{eq:arc_flow_conservation_constraints}
\text{s.t.: } & F_{\network, \varphi}(v) = \begin{cases} -z, &\text{ if } v = v^+, \\ z, &\text{ if } v = v^-, \\ 0, &\text{ otherwise}, \end{cases} & \forall v \in \nodeset,\\
\label{eq:arc_flow_side_constraints}
& \sum_{(u,v) \in {\arcset}} a_{(u,v)} \varphi_{(u,v)} \geq b, & \\
\label{eq:arc_flow_domain2}
& \varphi_{(u,v)} \in \mathbb{Z}_+, & \forall (u,v) \in {\arcset}, \\
\label{eq:arc_flow_domain1}
& z \in \mathbb{Z}, &
\end{align}

\noindent where $F_{\network, \varphi}(v) = \sum_{ (u,v) \in {\arcset}} \varphi_{(u,v)} - \sum_{ (v,w) \in {\arcset} } \varphi_{(v,w)}$.
Each variable $\varphi_{(u,v)}$ corresponds to the flow on arc $(u,v) \in \arcset$ and variable $z$ gives the total flow.

The flow decomposition theorem by Ahuja et al. \cite{AMO93} guarantees a one-to-many correspondence that preserves the objective value, between the primal space of the linear relaxation of path flow and arc flow models that are based on the same network. In this way, formulations \eqref{eq:path_flow_of}--\eqref{eq:path_flow_domain} and \eqref{eq:arc_flow_of}--\eqref{eq:arc_flow_domain1} have the same primal strength.
This allows us to handle primal solutions of both path flow and arc flow models equivalently.
In Section \ref{sec:linear_relaxation}, we also formalize a correspondence between the dual space of such models.

Examples of equivalent arc flow and path flow models derived from DW decompositions include the arc flow model in \cite{V99} and the pattern-based formulation in \cite{GG61,GG63} for the CSP. In the context of the capacitated vehicle routing problem, equivalent models are the $q$-route formulation by Christofides et al. \cite{CMT81} and the capacity-indexed formulation (see, e.g., \cite{FLLARUW06,PPU08}).
Although path flow and arc flow models based on the same network solve the same problem and have equally strong relaxations, each formulation may be preferable under certain specific aspects, which we contrast in the following.

{\bf Overall Solution.} The solution of path flow models typically relies on B\&P due to the exponential number of variables. State-of-the-art B\&P algorithms are usually specialized and contain problem-specific features (with the recent exception of \cite{PSUV20}, which solves a variety of problems). Arc flow models too can be solved by B\&P, but, due to their much smaller size, they may be solved in practice by MILP solvers, which is the most popular approach for these models.
Few authors explore the correspondence between path flow and arc flow models by combining them into a unique solution method. For instance, Pessoa et al. \cite{PUPR10} proposed a B\&P algorithm based on a path flow model that solves the problem at the root node as an arc flow model by a MILP solver if the network is sufficiently small.

{\bf Dealing with Exponential Networks.} In many strongly $\mathcal{NP}$-hard problems, the pricing only admits a decision network of exponential size (unless $\mathcal{P} = \mathcal{NP}$). Path flow models usually address the exponential nature of the network implicitly in specialized pricing algorithms that consider dominance and reduction criteria. On the other hand, practical solutions for arc flow models are limited to networks of at most pseudo-polynomial size. However, pseudo-polynomial arc flow models can be derived from exponential networks by relying on a state-space relaxation of the pricing (see, e.g., \cite{LACIV21}).

{\bf Linear Relaxation Solution.} The LP relaxation of path flow models is usually solved by CG. For arc flow, the LP relaxation of small- and medium-sized models may be efficiently solved by simplex/barrier algorithms. Still, larger models must rely on column-and-row generation (as in \cite{V99}). The basis of a path flow model is smaller as it does not consider flow conservation constraints. Thus, at each iteration, the restricted master problem is solved faster. In contrast, an arc flow basis is larger but less degenerate, so CG usually converges within fewer pricing iterations (see, e.g., \cite{LACIV21,SV13}). 

{\bf Branching and Cutting Planes.} It is well known that additional constraints based on arc flow variables are {\em robust}, i.e., they do not impact on the pricing structure. In contrast, additional constraints based on path flow variables are often non-robust and may heavily impact the pricing. In the context of network flow, this impact may cause an expansion of the network, increasing its complexity. Some authors solve path flow models by using robust branching and cutting planes based on arc flow variables (see, e.g., \cite{AV08,U11}). However, many state-of-the-art B\&P algorithms for path flow models still use non-robust branching constraints and cuts, since the improvement in terms of strength may pay off the increased pricing complexity. Then, a clear disadvantage of solving arc flow models with weak relaxations directly by a MILP solver is that they must rely solely on branching and cutting planes for general MILP, which may not be competitive when compared to the state-of-the-art for the equivalent path flow models. 

To summarize, path flow better addresses networks of exponential size and is also better suited for methods that rely on non-robust branching and/or cutting planes. On the other hand, arc flow is usually a good alternative to address problems with pseudo-polynomial networks and when the use of non-robust branching or cutting planes is not crucial, and their reduced size often allows the solution by MILP solvers.

\section{An Overview of the Solution Framework} \label{sec:overview}

In this section, we provide an overview of \framework. A depictive example is given in Figure \ref{fig:branching_tree}.
The input of \framework{} is a network $\network$, an upper bound $z_{ub}$ on the optimal integer solution value, the coefficients $c_{(u,v)}$ and $a_{(u,v)}$ for each arc $(u,v) \in \arcset$, and $b$. 
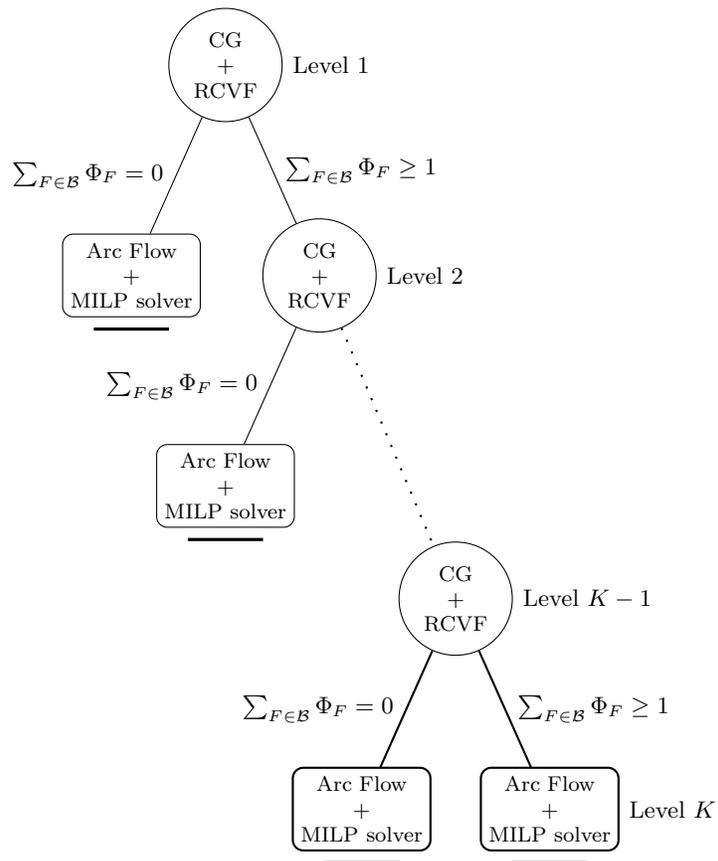
\begin{figure}[htb]
    \tikzset{thick,
         tree node/.style = {align=center, inner sep=0pt, font = \scriptsize},
every label/.append style = {font=\scriptsize},
                 S/.style = {draw, circle, rounded corners, minimum size = 11mm, inner sep=3pt,
                             top color=white, font = \scriptsize},
               ENL/.style = {
                             font=\footnotesize, left=1pt},
               ENR/.style = {
                             font=\footnotesize, right=1pt},
                     grow = down,
         sibling distance = 2.5cm,
           level distance = 2.8cm
           }
    \newcommand\LB{
                    \tikz\draw[very thick,solid] (-0.5,0) -- + (1,0);}

\centering
\begin{tikzpicture}
\node [S, align=center,label=0:{\footnotesize{Level 1}}] {CG\\+\\RCVF}
    child{node [S, rectangle, label=below:\LB, align=center] {Arc Flow\\+\\MILP solver}
        edge from parent node[ENL] {$ \sum_{F \in \mathcal{B}} \Phi_{F} = 0$}}
    child{node [S, align=center,label=0:{\footnotesize{Level 2}}] {CG\\+\\RCVF}
        child{node [S, rectangle, label=below:\LB,align=center] {Arc Flow\\+\\MILP solver}
            edge from parent node[ENL] {$ \sum_{F \in \mathcal{B}} \Phi_{F} = 0$}}
        child{node [S, align=center,below=0mm,right=-7mm,label=0:{\footnotesize{Level $K-1$}}] at (0.5,-1.5) {CG\\+\\RCVF}
			child{node [S, rectangle, label=below:\LB,align=center,solid] {Arc Flow\\+\\MILP solver}
            	edge from parent[solid] node[ENL] {$ \sum_{F \in \mathcal{B}} \Phi_{F} = 0$}}
			child{node [S, rectangle, label=below:\LB,align=center,solid,label=0:{\footnotesize{Level $K$}}] {Arc Flow\\+\\MILP solver}
            	edge from parent[solid] node[ENR] {$ \sum_{F \in \mathcal{B}} \Phi_{F} \geq 1$}}
            edge from parent[thick, loosely dotted] }
        edge from parent node[ENR] {$ \sum_{F \in \mathcal{B}} \Phi_{F} \geq 1$}
            };
\end{tikzpicture}
\caption{Example of a branching tree with $K$ levels.}
\label{fig:branching_tree}
\end{figure}
{The input is first used to build a path flow model of type \eqref{eq:path_flow_of}--\eqref{eq:path_flow_domain}.} 
The LP relaxation of the model is solved by the CG algorithm of Section \ref{sec:linear_relaxation}.
To remove arcs that do not improve the current incumbent, we use three RCVF strategies, each based on a different search for dual solutions (described in Section \ref{sec:variable_fixing}). These strategies are effective in those cases where the optimal solution is at hand and the challenge is to prove optimality. The first strategy is a traditional one based on the dual solution obtained at the end of the CG for the original LP relaxation. The second and third strategies are more expensive and are based on dual solutions obtained by solving specialized path flow and arc flow LP models. At the root node, we apply the first and second strategies after solving the LP relaxation.

\framework{} uses the branching scheme of Section \ref{sec:branching}, which is particularly effective when the incumbent is still not optimal. The tree is limited to $K$ levels. At each level, the left branch is obtained by a large elimination of arcs and is directly solved as an arc flow model by a MILP solver. Although the aim is to provide relatively easier problems in the left branch, in some cases these problems can be small but still hard enough and consume most of the overall time limit. However, no additional stopping criterion is used to deal with such cases. In the right branch, the linear relaxation with an additional branching constraint is solved, followed by the application of the first RCVF strategy. The right branch is branched again in the first $K-1$ levels or solved as an arc flow model by a MILP solver in the last level. In this last node, we apply the third (and most expensive) RCVF strategy before invoking the MILP solver. The tree is explored by breadth-first search, by prioritizing left branches, and no parallelism is implemented.
\section{On the Solution of the Linear Relaxation} \label{sec:linear_relaxation}

In hard instances of strongly $\mathcal{NP}$-hard problems, the network underlying a DWM is usually huge, and efficient methods for the associated network flow models typically relies on CG algorithms.
In such algorithms, the LP relaxation with a restricted set of columns, called {\em restricted master problem} (RMP), is iteratively solved.  At each iteration, an oracle solves the pricing problem to generate non-basic columns with negative reduced cost. The algorithm halts if no such column can be found. For a deeper discussion, see, e.g.,  \cite{DDS06,LD05}.

This section describes the solution of the LP relaxation of path flow and arc flow models in \framework. We use column(-and-row) generation, and the dual plays a key role. For that, consider the dual LP relaxation of \eqref{eq:path_flow_of}--\eqref{eq:path_flow_domain}:
\begin{align}
\label{eq:path_flow_dual_objective}
\allowdisplaybreaks
\max ~& b^\intercal \beta, &\\
\label{eq:path_flow_dual_constraints}
\text{s.t.: } & {a_p}^\intercal\beta \leq c_p, & \forall p \in \pathset,\\
\label{eq:path_flow_dual_domain}
& \beta \in \mathbb{R}_+^m,  &
\end{align}
\noindent where $\beta$ are the dual variables associated with constraints \eqref{eq:path_flow_constraints}. 
The following is the dual LP relaxation of \eqref{eq:arc_flow_of}--\eqref{eq:arc_flow_domain1}:
\begin{align}
\label{eq:arc_flow_dual_objective}
\allowdisplaybreaks
\max~ & b^\intercal\beta, & \\
\label{eq:arc_flow_dual_constraints1}
\text{s.t.:~} & - \alpha_u + \alpha_v +  {a_{(u,v)}}^\intercal  \beta \leq c_{(u,v)}, & \forall (u,v) \in \arcset, \\
\label{eq:arc_flow_dual_constraints2}
& \alpha_{v^+} - \alpha_{v^-} \leq 0, \\
\label{eq:arc_flow_dual_domain1}
& \alpha_{u} \in \mathbb{R}, & \forall u \in \nodeset,\\
\label{eq:arc_flow_dual_domain2}
& \beta \in \mathbb{R}_+^m, &
\end{align}

\noindent where $\alpha$ and $\beta$ are the dual variables associated with flow conservation constraints \eqref{eq:arc_flow_conservation_constraints} and side constraints \eqref{eq:arc_flow_side_constraints}, respectively. 

\subsection{A Dual Correspondence} \label{sec:dual_correspondence}

\def\betaPF{{\overline\beta^{}}}
\def\betaAF{{\overline\beta^{}}}
\def\alfaAF{{\overline\alpha}}

Recently, de Lima et al. \cite{LACIV21} showed how arc flow models provide a richer description of the dual space when compared to path flow models, which results in less degeneracy in the LP solution. The proof is based on the fact that dual constraints of a path flow model can be obtained by aggregating dual constraints of the equivalent arc flow model. Here, we further exploit this fact to present a correspondence between the dual space of these models.
\begin{lemma} \label{lemma:beta_equivalence_1}
Given a feasible solution $\left\langle \alfaAF, \betaAF \right\rangle$ of \eqref{eq:arc_flow_dual_objective}--\eqref{eq:arc_flow_dual_domain2}, it holds that $\betaAF$ is also feasible for \eqref{eq:path_flow_dual_objective}--\eqref{eq:path_flow_dual_domain}.
\end{lemma}
\begin{proof}
The proof follows by showing that $\betaAF$ satisfies constraints \eqref{eq:path_flow_dual_constraints}, i.e., $a_p^\intercal \betaAF \leq c_p$, for all $p \in \pathset$. Since $\left\langle \alfaAF, \betaAF \right\rangle$ is dual-feasible for the arc flow model, then by constraints \eqref{eq:arc_flow_dual_constraints1} it holds that $\alfaAF_v - \alfaAF_u + {a_{(u,v)}}^\intercal\betaAF  \leq c_{(u,v)}$, for every arc $(u,v) \in \arcset$, and by constraints \eqref{eq:arc_flow_dual_constraints2} it holds that $\overline \alpha_{v^+} - \overline\alpha_{v^-} \leq 0$. Then, it follows that, for each $p \in \pathset$, $c_p = \sum_{(u,v) \in \arcset_p} c_{(u,v)} \geq \sum_{(u,v) \in \arcset_p}(\overline\alpha_v - \overline\alpha_u + {a_{(u,v)}}^\intercal \betaAF) = \overline\alpha_{v^-} - \overline\alpha_{v^+} + \sum_{(u,v) \in \arcset_p} {a_{(u,v)}}^\intercal \betaAF \geq \sum_{(u,v) \in \arcset_p} {a_{(u,v)}}^\intercal \betaAF = {a_{p}}^\intercal\betaAF $. 
\end{proof}
\begin{lemma} \label{lemma:beta_equivalence_2}
Given a feasible solution $\betaPF$ of \eqref{eq:path_flow_dual_objective}--\eqref{eq:path_flow_dual_domain}, there exist $\alfaAF \in \mathbb{R}^{|\nodeset|}$ such that $\left\langle \alfaAF, \betaPF \right\rangle$ is feasible for \eqref{eq:arc_flow_dual_objective}--\eqref{eq:arc_flow_dual_domain2}.
\end{lemma}
\begin{proof}
Let $\alfaAF$ be defined by the following recursion:
\begin{align}
\label{eq:lemma_beta_equivalence_2_recursion}
\alfaAF_v = \begin{cases} \min \{ \alfaAF_u + (c_{(u,v)}-{a_{(u,v)}}^\intercal\betaPF) : (u,v) \in \arcset \}, & \text{ if } v \neq v^+, \\ 0, & \text{ if } v = v^+. \end{cases} 
\end{align}
In this definition, $\alfaAF_v$ represents the cost of a shortest path from $v^+$ to $v$, where the cost of each arc $(u,v) \in \arcset$ is given by $(c_{(u,v)}-{a_{(u,v)}}^\intercal\betaPF)$.
The proof follows by showing that $\alfaAF$ and $\betaPF$ satisfy the dual constraints \eqref{eq:arc_flow_dual_constraints1} and \eqref{eq:arc_flow_dual_constraints2}. Constraints \eqref{eq:arc_flow_dual_constraints1} are satisfied, for every arc $(u,v) \in \arcset$, because, by the definition of $\alfaAF_v$, it holds that $\alfaAF_v \leq \alfaAF_u + (c_{(u,v)}-{a_{(u,v)}}^\intercal\betaPF)$, which directly implies that $\alfaAF_v - \alfaAF_u + {a_{(u,v)}}^\intercal\betaPF  \leq c_{(u,v)}$. To show that constraint \eqref{eq:arc_flow_dual_constraints2} is satisfied, we first highlight that, since $\alfaAF_{v^-}$ corresponds to the cost of a shortest path from $v^+$ to $v^-$, then $ \alfaAF_{v^-} = \min \{ c_p - {a_p}^\intercal\betaPF : p \in \pathset\}$. Since $\betaPF$ is dual-feasible for the path flow model, then it follows by constraints \eqref{eq:path_flow_dual_constraints} that $c_p - {a_p}^\intercal\betaPF \geq 0$, for every $p \in \pathset$, which directly implies that $\alfaAF_{v^-} \geq 0$. Then, since $\alfaAF_{v^+} = 0$ and $\alfaAF_{v^-} \geq 0$, it follows that $\alfaAF_{v^+}-\alfaAF_{v^-} \leq 0$.
\end{proof}
By combining Lemmas \ref{lemma:beta_equivalence_1} and \ref{lemma:beta_equivalence_2}, the following result is derived:
\begin{theorem}[Dual Correspondence] \label{theorem:beta_equivalence}
There is a one-to-many correspondence between the solution space of \eqref{eq:path_flow_dual_objective}--\eqref{eq:path_flow_dual_domain} and \eqref{eq:arc_flow_dual_objective}--\eqref{eq:arc_flow_dual_domain2}, respectively, in which variables $\beta$ preserve the same solution on both sides of the mapping. 
\end{theorem}
Theorem \ref{theorem:beta_equivalence} guarantees that dual solutions $\overline\beta$ of either an arc flow or a path flow model may be used without distinction in the methods presented next.

\subsection{Computing the Minimum Reduced Cost of Arcs} \label{sec:minimum_rc_arcs}
Let us define the minimum reduced cost of an arc as follows:
\begin{definition}
Given a dual solution $\overline\beta \in \mathbb{R}_+^m$, the {\em minimum reduced cost of an arc} $(u,v) \in \arcset$ is the minimum reduced cost of a path among all paths that contain $(u,v)$, and is given by $\overline c_{(u,v)} = \min \{ c_p - {a_p}^\intercal \overline\beta  : p \in \pathset_{(u,v)} \}$.
\end{definition}

The minimum reduced costs of the arcs are used in the oracle presented in Section \ref{sec:column_generation} and in the RCVF strategies presented in Section \ref{sec:variable_fixing}.
Irnich et al. \cite{IDDH10} showed how to efficiently compute $\overline c_{(u,v)}$, for all $(u,v) \in \arcset$, by bidirectional Dynamic Programming (DP). First, one has to compute the values $\overline c^+_v$ and $\overline c^-_v$, for each node $v \in \nodeset$, which represent the minimum reduced cost associated with all paths from $v^+$ to $v$ and from $v$ to $v^-$, respectively. The fact that $\network$ is acyclic allows one to compute $\overline c^+_v$ and $\overline c^-_v$, for each $v \in \nodeset$, by means of the following recursions:
\begin{align}
\nonumber
\overline c^+_v =& \begin{cases} \min \{ \overline c^+_u + (c_{(u,v)}-{a_{(u,v)}}^\intercal \overline\beta) : (u,v) \in \arcset \}, & \text{ if } v \neq v^+, \\ 0, & \text{ if } v = v^+, \end{cases} \\
\nonumber
\overline c^-_v =& \begin{cases} \min \{ \overline c^-_w + (c_{(v,w)}-{a_{(v,w)}}^\intercal \overline\beta) : (v,w) \in \arcset \}, & \text{ if } v \neq v^-, \\ 0, & \text{ if } v = v^-. \end{cases}
\end{align}

Efficient implementations of a labelling DP algorithm based on the two recursions allows us to simultaneously compute $\overline c^+_v$ and $\overline c^-_v$, for all $v \in \nodeset$, in $O(\arcset)$ time complexity. Then, the minimum reduced cost of an arc ${(u,v)}$ is simply given by $\overline c_{(u,v)} = \overline c^+_u + (c_{(u,v)}-{a_{(u,v)}}^\intercal \overline\beta) + \overline c^-_v$. The recursive structure also allows us to efficiently retrieve a path associated with $\overline c_{(u,v)}$ after the execution of the DP algorithm.

\subsection{Path-Based Pricing in Arc Flow} \label{sec:arc_flow_pricing}

In CG algorithms for \eqref{eq:path_flow_of}--\eqref{eq:path_flow_domain}, given a dual solution $\overline\beta$, the pricing problem
\begin{align}
\label{eq:pricing_network_flow}
\min \{ c_p - {a_p}^\intercal \overline\beta  : p \in \pathset \}
\end{align}

\noindent can be solved as a shortest path problem on $\network$, with arc costs $c_{(u,v)} - {a_{(u,v)}}^\intercal \overline\beta$, for $(u,v) \in \arcset$. As networks from a DWM are acyclic, a shortest path can be  found in $O(|\arcset|)$ by a topological ordering of the nodes (see, e.g., \cite{AMO93}).
In arc flow models, to search for a negative reduced cost variable, one can solve
\begin{align}
\label{eq:pricing_arc_flow}
\min \{ c_{(u,v)} - (-\overline\alpha_u+\overline\alpha_v+ {a_{(u,v)}}^\intercal\overline\beta)  : (u,v) \in \arcset \},
\end{align}
which generates a single arc. In practice, however, generating complete paths (each to be decomposed in a set of arcs) may significantly improve convergence time (see \cite{V99}). In what follows, we formalize the correctness of this approach.
\begin{proposition} \label{proposition:arc_flow_pricing}
Let $\overline\alpha \in \mathbb{R}^{|\nodeset|}$ and $\overline\beta \in \mathbb{R}_+^m$ be a solution of \eqref{eq:arc_flow_dual_objective}--\eqref{eq:arc_flow_dual_domain2}, possibly infeasible but satisfying \eqref{eq:arc_flow_dual_constraints2}. Then, any path $p \in \pathset$ with negative reduced cost contains at least one arc $(u,v) \in \arcset_p$ with negative reduced cost.
\end{proposition}
\begin{proof}
Let $p \in \pathset$ be a path with reduced cost $c_p - {a_p}^\intercal \overline\beta < 0$. We prove that at least one arc in $\arcset_p$ has negative reduced cost by showing that the sum of the reduced costs of arcs in $\arcset_p$, given by $\sum_{(u,v) \in \arcset_p } (c_{(u,v)} - (-\overline\alpha_u+\overline\alpha_v+ {a_{(u,v)}}^\intercal\overline\beta))$, is negative. By properly canceling $\overline\alpha$ terms from intermediate nodes in $p$, the sum of reduced costs is equal to $\sum_{(u,v) \in \arcset_p } (c_{(u,v)} - {a_{(u,v)}}^\intercal\overline\beta) + (\overline\alpha_{v^+}-\overline\alpha_{v^-}) = (c_p - a_p^\intercal\overline\beta) + (\overline\alpha_{v^+}-\overline\alpha_{v^-})$. Since we suppose that $c_p - {a_p}^\intercal \overline\beta < 0$ and, by \eqref{eq:arc_flow_dual_constraints2}, $(\overline\alpha_{v^+}-\overline\alpha_{v^-}) < 0$, it follows that the sum of reduced costs of arcs in $\arcset_p$ is negative, and, therefore, at least one arc in $\arcset_p$ has negative reduced cost.
\end{proof}
Proposition \ref{proposition:arc_flow_pricing} ensures that in the solution of an arc flow model relaxation by CG, if the pricing problem is solved as \eqref{eq:pricing_network_flow} and all arcs in the resulting path are introduced in the RMP, then the CG algorithm only halts when all paths in $\pathset$ have non-negative reduced cost. Hence, the $\overline\beta$ solution at the end of the CG algorithm is feasible for \eqref{eq:path_flow_dual_objective}--\eqref{eq:path_flow_dual_domain}, and, by relying on Theorem \ref{theorem:beta_equivalence}, there exists an $\overline\alpha$, computed by \eqref{eq:lemma_beta_equivalence_2_recursion}, such that $\left\langle \overline\alpha, \overline\beta \right\rangle$ is feasible for \eqref{eq:arc_flow_dual_objective}--\eqref{eq:arc_flow_dual_domain2}.

\subsection{Column(-and-Row) Generation Algorithm} \label{sec:column_generation}

\framework{} adopts a CG algorithm for path flow models and a column-and-row generation algorithm for arc flow models. For path flow, all rows are included in the RMP from the beginning. For arc flow, instead, we generate flow conservation constraints on demand, as proposed in \cite{V99}, by adding a constraint \eqref{eq:arc_flow_conservation_constraints} only when the corresponding vertex $v$ first appears as head or tail of an arc associated with an RMP variable.

Both algorithms include in the initial RMP base artificial variables with sufficiently high objective cost, so as to avoid infeasibility. In addition, they reuse the optimal basis of previous LP relaxations, when available.

In many problems, generating a single column per pricing iteration may lead to a slow convergence of the algorithm. In network flow models, it may be preferable to generate multiple paths at each iteration. Therefore, we implemented an oracle that generates multiple paths with negative reduced cost (if any exists), and, by relying on the discussion in Section \ref{sec:arc_flow_pricing}, can be used to solve the pricing of both path flow and arc flow models.

The oracle that we developed begins by computing $\overline c_{(u,v)}$, for each $(u,v) \in \arcset$, by using the method described in Section \ref{sec:minimum_rc_arcs}. Then, for each row $k=1,\ldots,m$ in \eqref{eq:path_flow_constraints} or in \eqref{eq:arc_flow_side_constraints}, the algorithm selects an arc $(u,v)$ with minimum $\overline c_{(u,v)}$ among the arcs that cover  $k$ (i.e., that have non-zero coefficient on row $k$), and, by using the DP structure, it generates a path in $\pathset_{(u,v)}$ of minimum reduced cost. 
At the end, the oracle has generated, for each row $k$, a path of minimum reduced cost that covers $k$. For instance, the algorithm generates the columns of minimum reduced cost, such that each item (in packing problems) or each client (in routing problems) is covered by at least one of such columns.
Notice that different rows may lead to the same column, so repeated columns are discarded.
The overall algorithm has $O(m\zeta + |\arcset|)$ time complexity, where $\zeta$ is the maximum length of a path, and it is preferable to use it when the matrix of \eqref{eq:path_flow_constraints} is sparse, so that more different columns are generated. 

\subsection{Dealing with Dual Infeasibility} \label{sec:safe_dual}
To apply RCVF (Section \ref{sec:variable_fixing}) and to prune nodes by a dual bound in the search of an integer optimal solution, a dual-feasible solution is needed. Most state-of-the-art LP solvers work within the limited precision of floating-point arithmetic and produce solutions that are feasible within a small margin of error. In particular, the LP solver that we use {(Gurobi 9.1.1)} assumes a dual solution of a minimization problem to be feasible if its minimum reduced cost is greater than $-\epsilon$, where $\epsilon$ is at minimum $10^{-9}$. Then, when the minimum reduced cost lies between $-10^{-9}$ and $0$, the dual solution is still infeasible but cannot be improved by CG because the solver may not include new columns in the RMP base. This issue could be solved with the use of an LP solver with exact precision, but at the cost of a significant efficiency loss. Instead, we implemented a method that attempts transforming a dual-infeasible solution (with a minimum reduced cost between $-10^{-9}$ and $0$) into a feasible one.

A constraint in \eqref{eq:path_flow_constraints} is defined as a {\em covering} constraint if all of its left-hand-side coefficients are non-negative and as a {\em packing} constraint if all of its left-hand-side coefficients are non-positive.
The following result shows how to efficiently derive dual-feasible solutions for a large class of models:
\begin{proposition} \label{proposition:dual_transformation_1}
Consider a path flow model \eqref{eq:path_flow_of}-\eqref{eq:path_flow_domain} where the first $1,\ldots,m'$ constraints are covering constraints and the last $m'+1,\ldots,m$ constraints are packing constraints. Let $\epsilon > 0$, and $\overline \beta \in \mathbb{R}^m_+$ be an (infeasible) dual solution with minimum reduced cost of a path between $-\epsilon$ and $0$. If $c_p$ is a multiple of $\epsilon$, for each $p \in \pathset$, then, $\overline \beta' = (\epsilon \lfloor \epsilon^{-1} \overline\beta_1 \rfloor, \ldots, \epsilon \lfloor \epsilon^{-1} \overline\beta_{m'} \rfloor, \epsilon \lceil \epsilon^{-1} \overline\beta_{m'+1} \rceil,\ldots, \epsilon \lceil \epsilon^{-1}\overline\beta_m \rceil)$ is a dual-feasible solution.
\end{proposition}
\begin{proof}
Let $p\in \pathset$ be an arbitrary path and $\overline c_p(\overline\beta)$ and $\overline c_p(\overline\beta')$ be the reduced cost of $p$ associated with the dual solutions $\overline\beta$ and $\overline\beta'$, respectively.  We prove that $\overline\beta'$ is dual-feasible by showing that $\overline c_p(\overline\beta') \geq 0$.
By considering that $a_{pk} \geq 0$ and $\overline\beta'_k \leq \overline\beta_k$, for each $k=1,\ldots,m'$, and that $a_{pk} \leq 0$ and $\overline\beta'_k \geq \overline\beta_k$, for each $k=m'+1,\ldots,m$, it is easy to check that $\overline c_p(\overline\beta') \geq \overline c_p(\overline\beta)$, which directly implies that $\overline c_p(\overline\beta') > -\epsilon$, since $\overline c_p(\overline\beta) > -\epsilon$. Now, considering that $c_p$ is a multiple of $\epsilon$ and that $a_{pk}$ is integer, for all $k=1,\ldots,m$, we have that $\overline c_p(\overline\beta')$ is also a multiple of $\epsilon$. Finally, the fact that $\overline c_p(\overline\beta')$ is a multiple of $\epsilon$ and is strictly greater than $-\epsilon$ implies that $\overline c_p(\overline\beta') \geq 0$.
\end{proof}
Proposition \ref{proposition:dual_transformation_1} is a generalization of the technique by Held et al. \cite{HCS12} to obtain safe dual-bounds for the vertex coloring problem. We consider $\epsilon = 10^{-9}$, which is always a divisor of the integer coefficients $c_p$. Hence, in our algorithm, the proposition can always be applied to models where constraints are only of covering and/or packing type. By relying on Theorem \ref{theorem:beta_equivalence}, Proposition \ref{proposition:dual_transformation_1} can also be applied in the case of arc flow models. 

When \eqref{eq:path_flow_constraints} has additional constraints which are not covering nor packing, Proposition \ref{proposition:dual_transformation_1} cannot be applied, and we use a heuristic to try to obtain a dual-feasible solution by treating these additional constraints as if they were of covering type, and then applying the proposition to obtain $\beta'$. In theory, the heuristic procedure may fail to obtain a dual-feasible solution, and in such a case \framework{} would continue the optimization without the ability to perform RCVF or prune the current node by dual bound. In practice, in our computational experiments the only model which did not follow the hypothesis of Proposition \ref{proposition:dual_transformation_1} is the model for the two-stage guillotine cutting stock problem (section \ref{sec:applications_2gcsp}), but for this problem the heuristic procedure was always sufficient to provide a dual-feasible solution at the end of the CG.

\section{Variable-Fixing Based on Reduced Costs} \label{sec:variable_fixing}

\def\dualbound{$b^\intercal\overline\beta$}
\def\rdualbound{$\lceil b^\intercal\overline\beta \rceil$}

Reduced-cost variable-fixing is a well-known technique used to filter the domain of integer variables in general MILP models (see, e.g., \cite{HMS06} and page 389 of \cite{NW88}). 
Irnich et al. \cite{IDDH10} showed how the dual solution of a path flow model can be used in an RCVF algorithm to remove arcs.
The overall idea is:
\begin{rcvf_algorithm}
Given a primal bound $z_{ub}$ and a dual-feasible solution $\overline\beta$ of objective value $z_{lb}=$ \dualbound, first compute the minimum reduced cost $\overline c_{(u,v)}$ associated with each arc $(u,v) \in \arcset$ (e.g., as described in Section \ref{sec:minimum_rc_arcs}), and then remove from $\network$ every arc $(u,v) \in \arcset$ such that $\overline c_{(u,v)} > z_{ub}-z_{lb}-1$. 
\end{rcvf_algorithm}

\noindent An equivalent approach has been implemented by Bergman et al. \cite{BCH15} in the general context of Lagrangian bounds for multivalued decision diagrams.

Different dual solutions may correspond to different reduced costs, and this directly affects the effectiveness of the RCVF.
For instance, let $\theta_{(u,v)} = c_{(u,v)} - (-\alpha_u+\alpha_v+ {a_{(u,v)}}^\intercal\beta)$ denote the reduced cost of $(u,v) \in \arcset$ in model \eqref{eq:arc_flow_dual_objective}--\eqref{eq:arc_flow_dual_domain2}.
A dual solution which allows the removal of $(u,v)$ by RCVF is one that satisfies $b^\intercal \beta + \theta_{(u,v)} \geq z_{ub}$. If such a solution exists, it can be obtained by solving  the dual arc flow model with a modified objective:
\begin{align}
\label{eq:arc_flow_dual_rcvf_opt}
\max \{ b^\intercal\beta + \theta_{(u,v)} : \eqref{eq:arc_flow_dual_objective}, \eqref{eq:arc_flow_dual_constraints1}, \eqref{eq:arc_flow_dual_constraints2}, \eqref{eq:arc_flow_dual_domain1}, \eqref{eq:arc_flow_dual_domain2} \}.
\end{align}

\noindent However, optimizing a model tailored for each arc can be very time consuming in practice. In the general context of non-convex mixed-integer nonlinear programming, a similar concern has been addressed in optimization-based bound tightening (see, e.g., \cite{GBMW17}). In the particular context of MILP, Bajgiran et al. \cite{BCR17} proposed a method to maximize the number of variables fixed by RCVF, by solving a single MILP model derived from an extension of the dual linear relaxation of the original model. This MILP model has an additional binary variable for each original variable, indicating whether the resulting dual solution is able to eliminate the associated original variable by RCVF. In our case, this method is generally impractical, because we are concerned with models with a huge number of variables that must rely on CG.

Here, we propose two strategies to search for dual solutions that may
improve the RCVF effectiveness, motivated by:
\begin{remark} \label{remark:rcvf_basis_violation}
The removal of arcs with a positive primal value in the current optimal basis leads to a basis-violation, which strengthens the linear relaxation.
\end{remark}

In practice, this strengthening usually increases the effectiveness of further variable-fixing.
However, as shown in the following section, such arc-removal by RCVF can only be obtained by relying on sub-optimal dual solutions.

\subsection{Sub-Optimal Dual Solutions in Variable-Fixing}
Sub-optimal dual solutions often produce a greater effectiveness in RCVF. This counter-intuitive fact was already observed almost two decades ago, for instance, by Sellmann \cite{S04} in the context of constraint programming-based Lagrangian relaxations. Here we extend this observation by the following:
\begin{proposition} \label{theorem:suboptimal_rcvf}
Let $(u,v) \in \arcset$ be any arc with $\overline\varphi_{(u,v)} > 0$ in a given primal-optimal solution $\overline\varphi$.
Then, there does not exist any dual-optimal solution $\overline\beta$ such that $(u,v)$ can be removed by the variable-fixing algorithm based on $\overline\beta$ and a primal bound $z_{ub} \geq$ \rdualbound $+1$.
\end{proposition}

\begin{proof}
We assume that $z_{ub} \geq$ \rdualbound $+1$, since in the trivial case in which $z_{ub} =$ \rdualbound, optimality is already proven. By the classical complementary slackness theorem, it follows that if $\overline\varphi$ and $\overline\beta$ are, respectively, primal-optimal and dual-optimal solutions, then, either $\overline\varphi_{(u,v)} = 0$, or $\overline c_{(u,v)} = 0$. Then, by supposing that $\overline\varphi_{(u,v)} > 0$, we have that $\overline c_{(u,v)} = 0$. Consequently, \dualbound $+ \overline c_{(u,v)} = $ \dualbound $ \leq $ \rdualbound $ \leq z_{ub} -1$, which implies that $(u,v)$ cannot be removed by the variable-fixing algorithm based on $\overline\beta$ and $z_{ub}$. 
\end{proof}

By disregarding the trivial case in which $z_{ub}=$ \rdualbound, Proposition \ref{theorem:suboptimal_rcvf} ensures that, to remove arcs with a positive primal value in the current optimal basis, we need a dual solution that is sub-optimal. A generalization of Proposition \ref{theorem:suboptimal_rcvf} to the case of general MILP models (not restricted to network flow) is straightforward.

\subsection{Variable-Fixing Strategies}
In this section, we present the three RCVF strategies adopted in \framework. In a preliminary version of this paper \cite{LIM21}, we used a different strategy to obtain alternative dual solutions. This strategy is not presented here, because in our experiments it was always outperformed by the second and third strategies presented below.
The three adopted strategies differ mainly in the way in which a dual solution is computed. In all strategies, each dual solution obtained is used as input to the variable-fixing algorithm.

{\bf First strategy.} The traditional strategy based on a dual solution obtained at the end of CG for the original linear relaxation.

{\bf Second strategy.} We obtain dual solutions by solving a dual path flow model in which, motivated by Proposition \ref{theorem:suboptimal_rcvf}, the solution is allowed to be sub-optimal and the reduced-cost of paths with positive value on a primal-optimal solution are included in the objective. For that, we first consider the primal-optimal solution $\overline \lambda$ in terms of path flow variables. Then, the dual path flow model to be solved is given by: 
\begin{align}
\label{eq:dual_strategy_1_objective}
\max ~& b^\intercal \beta + \sum_{p \in \pathset : \overline\lambda_p > 0} \theta_p, &\\
\label{eq:dual_strategy_1_constraints_1}
\text{s.t.: } & {a_p}^\intercal\beta + \theta_p =  c_p, & \forall p \in \pathset, \overline\lambda_p > 0, \\
\label{eq:dual_strategy_1_constraints_2}
& {a_p}^\intercal\beta \leq c_p, & \forall p \in \pathset, \overline\lambda_p = 0, \\
\label{eq:dual_strategy_1_constraints_3}
& b^\intercal \beta \geq z_{lb} - \epsilon, & \\
& \beta \in \mathbb{R}_+^m,  & \\
\label{eq:dual_strategy_1_domain_3}
& \theta_p \in \mathbb{R}_+,  & \forall p \in \pathset, \overline\lambda_p > 0,
\end{align}
\noindent in which each $\theta_p$ explicitly models the reduced cost of variable $\lambda_p$, with $\overline\lambda_p > 0$. Constraints \eqref{eq:dual_strategy_1_constraints_3} allow the resulting dual solution to be sub-optimal for the original linear relaxation, but limited by a given $\epsilon > 0$.
Model \eqref{eq:dual_strategy_1_objective}--\eqref{eq:dual_strategy_1_domain_3} is solved by row generation (i.e., CG of its primal), by considering the primal variables associated with \eqref{eq:dual_strategy_1_constraints_1} and \eqref{eq:dual_strategy_1_constraints_3} already in the initial RMP. Notice that, when the variable-fixing successfully removes a basic arc, model \eqref{eq:dual_strategy_1_objective}--\eqref{eq:dual_strategy_1_domain_3} possibly provides a different dual solution if solved again. Then, we iteratively solve \eqref{eq:dual_strategy_1_objective}--\eqref{eq:dual_strategy_1_constraints_3} and use the resulting dual solution in the variable-fixing algorithm, until no more arcs are removed.

{\bf Third strategy.} This is a refinement (motivated by Remark \ref{remark:rcvf_basis_violation}) of the expensive method that solves a model \eqref{eq:arc_flow_dual_rcvf_opt} tailored for each arc. It considers the primal-optimal solution $\overline \varphi$ in terms of arc flow variables and solves a restricted sequence of models \eqref{eq:arc_flow_dual_rcvf_opt}, each tailored for an arc $(u,v) \in \arcset$, in which $0 < \overline\varphi_{(u,v)} < 1$. Since we expect an arc $(u,v)$ with $\overline\varphi_{(u,v)}$ closer to $0$ to be more likely to be removed by variable-fixing, we follow a non-decreasing order of $\overline\varphi_{(u,v)}$ to build the sequence of models to be solved. Each model is solved by column-and-row generation, producing a dual-feasible solution, which in turn is used as input to the variable-fixing algorithm. Although we solve only a restricted set of dual models, this strategy can still be expensive, and stopping criteria should be considered. In \framework, the strategy halts once the remaining number of arcs is smaller than $50 m$ (where $m$ is the number of side constraints).
Despite the high computational cost, this strategy can be very effective for cases in which the incumbent solution is optimal and the challenge is to prove optimality. Hence, we only use it on the most expensive node in our branch-and-price tree, i.e., the right branch of the last level. 

\section{Branching Scheme} \label{sec:branching}

A major issue in B\&P is that efficient branching schemes are not always robust. A number of works propose general branching schemes that minimize the impact on the pricing problem and at the same time help convergence to optimality (see, e.g., \cite{V10}). However, not many B\&P schemes exploit the network flow representation of a DWM (as done, e.g., in \cite{AV08}). In general, any constraint $\sum_{(u,v) \in \arcset} a'_{(u,v)}\varphi_{(u,v)} \geq b'$ based on a linear combination of arc flow variables impacts on a pricing solved as a shortest path problem by simply incrementing ${a'_{(u,v)}}^\intercal \beta'$ to the cost of each arc $(u,v) \in \arcset$, where $\beta'$ is the dual solution related to the constraint.
A consequent result is that branching rules based solely on arc flow variables are robust. 

Based on the primal correspondence given by the flow decomposition theorem, any arc flow variable $\varphi_{(u,v)}$ can be represented as a sum $\sum_{p \in \pathset_{(u,v)}} \lambda_p$ of variables from the equivalent path flow model.
Consequently, any arc flow constraint $\sum_{(u,v) \in \arcset} a'_{(u,v)}\varphi_{(u,v)} \geq b'$ can be directly represented as a path flow constraint $\sum_{(u,v) \in \arcset} a'_{(u,v)} \sum_{p \in \pathset_{(u,v)}} \lambda_{p} \geq b'$. 
On the other hand, it is not always possible to rewrite a linear constraint based on path flow variables in terms of arc flow variables. This further motivates branching rules based on arc flow variables, as the resulting branching constraints can be easily handled by both path flow and arc flow models.

\subsection{Proposed Branching Scheme} \label{sec:variable_selection_method}

The branching scheme we propose is based on sets of arcs and exploits the potential of a general MILP solver in finding optimal solutions of small/medium-sized models. Initially, arcs that share mutual characteristics are grouped into subsets. It then considers arcs in all subsets having null flow in the optimal linear solution of a model, and either eliminates all of them or forces at least one of them to be in the solution.

We define an {\em arc family} $\bfamily \subset 2^{\arcset}$ as an arbitrary set of mutually disjoint subsets of $\arcset$.
For each $F \in \bfamily$, let variable $\Phi_F~=~\sum_{(u,v)\in F} \varphi_{(u,v)}$ represent the aggregated sum of arc flow variables associated with arcs in $F$. Given a primal solution $\overline{\varphi}$ in terms of arc flow variables, we represent as $\overline\Phi_F = \sum_{(u,v) \in F} \overline\varphi_{(u,v)}$ the cumulated sum of the solution values of arcs in $F$.

The variable selection considers all variables related to the set of arcs $\mathcal{B} =  \{ (u,v) \in F \in \bfamily : \overline\Phi_F = 0\}$, which are used to create two branches. In the {\em left branch}, we implicitly consider the branching constraint
\begin{equation}
\sum_{(u,v) \in \mathcal{B}} \varphi_{(u,v)} = 0
\end{equation}
by removing all arcs in $\mathcal{B}$ from $\network$. In the {\em right branch} we add the constraint 
\begin{equation} \label{eq:right_branch_constraint}
\sum_{(u,v) \in \mathcal{B}} \varphi_{(u,v)} \geq 1
\end{equation}
to the model, implying that at least one arc in $\mathcal{B}$ must be in the solution. Depending on the definition of $\bfamily$, the left branch is expected to lead to a great reduction in the size of the network, while keeping variables that should provide good feasible solutions. The reduced problem may be solved by an alternative method. In particular, we conclude this branch by solving the residual arc flow model by a MILP solver. The domain reduction in the right branch may be weaker, but the branching constraint may behave as a cutting plane, hopefully improving the optimal dual bound. In fact, although no arcs are explicitly removed from the network in the right branch, the strengthening in the relaxation may improve the effectiveness of RCVF.

In Section \ref{sec:applications}, we present the arc families used in our applications to C\&P problems. Further examples for other applications can be found in \cite{LIM21}. 

\subsection{Lifting the Right-Branch Constraint} \label{sec:lifting_constraint}
Constraint \eqref{eq:right_branch_constraint} imposes that at least one arc in $\mathcal{B}$ must be in the solution. The network structure can be exploited to determine redundant arcs in $\mathcal{B}$ that can be in a solution only if other arcs in $\mathcal{B}$ are also in the solution. The use of any of such redundant arcs implies that the left-hand side of constraint \eqref{eq:right_branch_constraint} is at least $2$. Thus, even if a redundant arc is removed from $\mathcal{B}$ the set of integer solutions that are feasible w.r.t. \eqref{eq:right_branch_constraint} does not change. On the other hand, such removal may violate fractional solutions, and this strengthens the relaxation of the resulting model.

Once a redundant arc is removed from $\mathcal{B}$, the list of remaining redundant arcs must be updated, since some redundant arcs may become non-redundant.
In this way, a greedy removal of redundant arcs is not necessarily optimal w.r.t the number of arcs removed from $\mathcal{B}$. 
For that, we implemented a heuristic DP approach with $O(|\arcset|)$ time complexity to maximize the number of redundant arcs removed. The algorithm iterates over a topological ordering of $\nodeset$. At the iteration of node $u^*$, it computes whether there exists any path from $v^+$ to $u^*$ that does not have an arc in $\mathcal{B}$. If true, it proceeds to the next iteration, otherwise, it sets $\mathcal{B} \gets \mathcal{B} \setminus \{(u^*,v) \in \mathcal{B} \}$. To further remove redundant arcs from $\mathcal{B}$, we also apply the algorithm in the reversed network, obtained by inverting the direction of the arcs.

\section{Applications to Cutting and Packing Problems} \label{sec:applications}

We apply \framework{} to four C\&P problems that allow pseudo-polynomial arc flow models with strong relaxations. 
In the first three applications, we considered network flow models from the literature, which for conciseness are not explicitly reported here. However, since the arc flow model for the fourth application is new, we give its coefficients to model \eqref{eq:path_flow_of}--\eqref{eq:path_flow_domain} and \eqref{eq:arc_flow_of}--\eqref{eq:arc_flow_domain2}.

\subsection{Cutting Stock Problem} \label{sec:applications_csp}

In the CSP, we are given an unlimited number of stock rolls of length $W \in \mathbb{Z}_+$ and a set $I$ of items, each $i \in I$ associated with a width $w_i \in \mathbb{Z}_+$ and a demand $d_i \in \mathbb{Z}_+$. The objective is to cut the minimum number of stock rolls to obtain all demands. An equivalent problem is the bin packing problem (BPP), where $d_i = 1$, for all $i \in I$. We refer to \cite{DIM16} for a recent survey. 

The classical pattern-based model in \cite{GG61,GG63} is a path flow model derived from a DW decomposition of the textbook CSP model (see, e.g., \cite{MT90}). The  underlying network of the model in \cite{GG61,GG63} is a DP network of an unbounded knapsack problem. The equivalent arc flow model was first solved in \cite{V99}. The network of this model discretizes the stock roll into $W$ unitary positions, and each node in $ \nodeset \subseteq \{ v : v=0,\ldots,W \}$ is associated with a position, where $v^+ = 0$ and $v^- = W$. Each item $i \in I$ has a set $\arcset_i \subseteq \{ (v, v+w_i) : v \in \nodeset, v \leq W-w_i \}$ of arcs, where each arc $(v, v+w_i) \in \arcset_i$ represents the cut of item $i$ starting from $v$ in a stock roll. An additional set of arcs $\arcset^- = \{(v, v^-) : v \in \nodeset \}$ represents waste portions of a stock roll. The overall set of arcs is  $\arcset = \cup_{i \in I} \arcset_i \cup \arcset^-$.

A subset relation is used in the definition of $\nodeset$ and $\arcset_i$ ($i \in I$) since not all positions are necessary to solve the model exactly. Indeed, \cite{V99} proposed reduction criteria to remove redundant arcs by considering that items can always be cut following a non-increasing width ordering. Later, C{\'o}t{\^e} and Iori \cite{CI18} proposed the {\em meet-in-the-middle} (MIM) patterns, which allowed to produce significantly smaller networks.
Based on the following remark, we developed a technique that further reduces the network from \cite{V99} and may lead to smaller networks  than the ones resulting from the MIM patterns. 
\begin{remark} \label{property:csp1}
Given a CSP instance, let $\overline W \in \mathbb{Z}_+$ be a value ensuring that there exists an optimal solution where the maximum waste of a single stock roll is at most $\overline W$.
Then, all arcs contained only in paths associated with cutting patterns with a waste larger than $\overline W$ can be removed from the network.
\end{remark}
A straightforward way to compute $\overline W$ considers that in any CSP solution with $K$ stock rolls the maximum waste on each roll is at most $KW-\sum_{i\in I} w_id_i$. 
{All arcs that only lead to cutting patterns with a waste larger than $\overline W$} are computed by a back propagation in the network, similarly to the method by Trick \cite{T03} to propagate knapsack networks in constraint programming. The computation of the {\em waste-limited network} is given in Algorithm \ref{alg:csp_network}. By considering items ordered by non-increasing $w_i$, lines \ref{algline:csp_vdc_inicio} to \ref{algline:csp_vdc_fim} create the standard network in \cite{V99}, and lines \ref{algline:csp_reduction_inicio} to \ref{algline:csp_reduction_fim} impose the reduction from Remark \ref{property:csp1}.

\begin{algorithm}
\DontPrintSemicolon 
\KwIn{$W$, $I$, and $\overline W$}
$\nodeset^+ \gets \{0\}$\; \label{algline:csp_vdc_inicio}
\For {$i=1,\ldots,|I|$}{
	\For{$copy=1,\ldots,d_i$}{
		\For{$v\in \nodeset^+$ in decreasing order}{
			\If{$v+w_i \leq W$}{
				$\arcset_i \gets \arcset_i \cup \{ (v, v+w_i) \} $\;
				$\nodeset^+ \gets \nodeset^+ \cup \{ v+w_i \} $\; \label{algline:csp_vdc_fim}
			}
		}
	}
}
$\nodeset^- \gets \{W-\overline W,\ldots,W\}$\; \label{algline:csp_reduction_inicio}
\For {$i=1,\ldots,|I|$}{
	\For{$copy=1,\ldots,d_i$}{
		\For{$(v,v+w_i) \in \arcset_i$ in increasing order of $v$}{
			\If{$v+w_i \in \nodeset^-$}{
				$\nodeset^- \gets \nodeset^- \cup \{ v \} $\;
			}
			\Else{
				$\arcset_i \gets \arcset_i \setminus \{ (v, v+w_i) \} $\; \label{algline:csp_reduction_fim}
			}
		}
	}
}

$\nodeset \gets \{v : (v, v+w_i) \in \arcset_i \text{ or } (v-w_i, v) \in \arcset_i, i \in I \} \cup \{0, W\}$\;
$\arcset^- \gets \{ (v, W) : v \geq W-\overline W, v \in \nodeset \}$\;
\Return{$\nodeset$, $\arcset_i$ (for all $i\in I$) and $\arcset^-$}\;
\caption{{\sc WasteLimitedNetworkCSP}}
\label{alg:csp_network}
\end{algorithm}

We consider two general classes of arc families. The first arc family
\[ \bfamily^a_k = \{ F_n = \{ (u,v) \in \arcset : u \in \{ nk,\ldots,nk+(n-1)\} \} : n = 0,\ldots,\lceil W/k \rceil - 1 \} \]
\noindent considers the sets of nodes sequentially partitioned in $\lceil W/k \rceil -1$ parts of up to $k$ nodes each. Then, each set $F_n$ represents the $n$-th part, and it contains all arcs whose tail lies in such part. In general, larger values of $k$ provide a conservative reduction in the left branch. The second arc family
\[ \bfamily^b_k = \{ F_{in} = \{ (u,v) \in \arcset_i : u \in \{ nW/k,\ldots,(n+1)W/k-1\} \} : n = 0,\ldots,k-1 \} \]
\noindent considers the set of arcs of each item partitioned into $k$ parts, following an increasing order of the nodes. Each set $F_{in}$ contains each arc in $\arcset_i$ whose tail lies in the $n$-th part. For instance, when $k=2$, $\bfamily^b_2$ represents a partition of the arcs of each item into two parts, each representing either the first half or the second half of a stock roll. If $k=W$, then each part contains up to a single arc. For this family, larger values of $k$ generate in the left branch problems that are smaller, but also less likely to contain good integer solutions.

\subsection{Two-Stage Guillotine Cutting Stock Problem} \label{sec:applications_2gcsp}
In the two-stage guillotine cutting stock problem (2GCSP), we are given an unlimited number of two-dimensional stock sheet with width $W$ and height $H$, and a set $I$ of two-dimensional items. Each item $i \in I$ has width $w_i$, height $h_i$, and demand $d_i$. The aim is to cut all items  from the minimum number of stock sheets, by using two-stage guillotine cuts. The first and second cut stages consist of, respectively, horizontal and vertical cuts parallel to the edges of the stock sheet. To separate items from waste, a third trimming stage is allowed.

The state-of-the-art exact method for the 2GCSP is the arc flow model in \cite{CI18}, which corresponds to the model proposed by Macedo et al. \cite{MAV10} enhanced by the use of the MIM patterns.
Let $H^* = \{ h_i : i \in I \}$ be the set of all different item heights.
The arc flow model has a graph $(\nodeset^1, \arcset^1)$ representing first-stage cut decisions, in which
$\nodeset^1  \subseteq \{ v : v = 0,\ldots,H \}$ and $\arcset^1  \subseteq \{ (v,v+h) : v \in \nodeset^1, h \in H^*  \}$.
There is also a graph $(\nodeset^2_h, \arcset^2_{h})$ representing second-stage cut decisions in each strip of height $h \in H^*$ cut in the first stage, in which
$\nodeset^2_h	 \subseteq \{ v : v = 0,\ldots,W \}$ and $\arcset^2_{h}  \subseteq \{ (v,v+w_i) : v=0,\ldots,W, i \in I, h_i \leq h \}$.
There are also source arcs connecting $v^+$ to $0$ in $\nodeset^1$ and $\nodeset^2_h$ (for each $h \in H^*)$ and sink arcs connecting each node in $\nodeset^1$ and $\nodeset^2_h$ (for each $h \in H^*$) to $v^-$. A full description of the model is reported in \cite{CI18}.

In our experiments for the 2GCSP, we solve the arc flow model in \cite{CI18} with \framework{} by setting $K=10$ and using the  arc family:
\begin{align*}
\bfamily = & \{ F_v = \{ (v, v+h) \in \arcset^1_h : h \in H^* \} : v \in \nodeset^1 \} \\
& \cup \{ F_v = \{ (v,v+w_i) \in \arcset^2_{h} : h \in H^*, i \in I, h_i \leq h\} : v \in \cup_{h \in H^*} \nodeset^2_h \},
\end{align*}
\noindent where the first and second parts correspond to all arcs related to specific first-stage and second-stage cut positions, respectively.

Since the arc flow model for the 2GCSP does not follow the hypothesis of Proposition \ref{proposition:dual_transformation_1}, we must rely on the heuristic method in Section \ref{sec:safe_dual} to convert dual-infeasible solutions into feasible ones, which may fail. We attested that this heuristic conversion always succeeded in the solution of the LP relaxation in all nodes of the branching tree. However, it did not always succeed for the (many) dual-infeasible solutions obtained in the second and third variable-fixing strategies, which were consequently deactivated.

\subsection{Skiving Stock Problem} \label{sec:applications_ssp}
In the skiving stock problem (SSP), we are given the minimum width $W$ of a large object and a set $I$ of items, where each item $i\in I$ has a width $w_i$ and a maximum number of copies $b_i$. The objective is to recompose the items into the maximum number of large objects.

The state-of-the-art exact method for the SSP is the reflect formulation in \cite{MDISS20}. In our experiments, we solve the pseudo-polynomial arc flow model in \cite{MS16a}. The network of this model uses a set $\nodeset \subseteq \{ v : v=0,\ldots,W' \}$ of nodes to represent integer linear combinations of the given item widths, where $W' \geq W$ is a sufficiently large value, $v^+ = 0$, and $v^- = W'$. A set $\arcset_i \subseteq \{ (v, v+w_i) : v \in \nodeset, v \leq W'-w_i \}$ of arcs is associated with each $i \in I$, where each arc $(v, v+w_i) \in \arcset_i$ represents the inclusion of an item $i$ from positions $v$ to $v+w_i$ of a large object. There is also a set of sink arcs $\arcset^- \subseteq \{(v, v^-) : v \in \nodeset, v \geq W\}$, where each $(v, v^-)$ represents the final composition of a large object with total width $v$ in the solution. The overall set of arcs is  $\arcset = \cup_{i \in I} \arcset_i \cup \arcset^-$. 

In our experiments, the arc flow model is solved by \framework{} based on an arc family equivalent to $\bfamily^b_{20}$ (Section \ref{sec:applications_csp}) and with $K=10$.

\subsection{Ordered Open-End Bin Packing Problem} \label{sec:applications_ooebpp}
In the ordered open-end bin packing problem (OOEBPP), we are given an unlimited number of copies of a one-dimensional bin with capacity $W$ and a sequence $I = (1,\ldots,m)$ of one-dimensional items, where each item $i \in I$ has weight $w_i$. The aim is to pack all items into the minimum number of bins, by allowing the last item (and only the last item) of the sequence in each bin to exceed the bin capacity.

The state-of-the-art exact method for the OOEBPP is a B\&P algorithm for a set-covering formulation (hence a path flow model) by Ceselli and Righini \cite{CR08}. 
We use the model in \cite{CR08} to derive an equivalent arc flow model where the network is based on a pseudo-polynomial pricing algorithm. This network is similar to the one of the DP-flow model for the CSP (see, e.g., \cite{DIM16}), which is based on the classical DP recursion for the knapsack problem. The set of nodes is given by $\nodeset \subseteq \{ (i, W') : i = 1,\ldots,m, W'=0,\ldots,W-1 \} \cup \{ v^- \}$, where $v^+ = (1,0)$. Each item $i=1,\ldots,m-1$ has a set $\arcset_i \subseteq \{ ((i,W'), (i+1, W'+w_i)) : W'=0,\ldots,W-w_i-1\}$ of arcs representing its selection as a non-last item in a bin, and a set $\arcset^{d}_i \subseteq \{ ((i,W'), (i+1,W')) : W'=0,\ldots,W-1  \}$ of dummy arcs representing the decision of not taking $i$ item in a path. Each item $i=1,\ldots,m$ is associated with a set $\arcset^-_i \subseteq \{ ((i,W'), v^-) : W'=0,\ldots,W-1\}$, of sink arcs representing the decision of taking $i$ as the last item in a bin. For the last item ($i=m$) in the input, we define $\arcset_m = \arcset^d_m = \emptyset$. The full set of arcs is given by $\arcset = \cup_{i \in I}(\arcset_i \cup \arcset^d_i \cup \arcset^-_i)$. An advantage of this network is that it explicitly models the ordering of the items in all feasible paths. 

In the resulting arc flow model, the objective function minimizes the number of arcs reaching the sink node, i.e., the number of paths (bins) used. In this way, $c_{(u,v)}$ is $1$, if $(u,v) \in \cup_{i \in I} \arcset^-_i$, and $0$, otherwise. The side constraints guarantee that each item is included in at least one path (bin). Thus, for all $i \in I$, $b_i = 1$ and $a_{i(u,v)} = 1$, for all $(u,v) \in \arcset_i \cup \arcset^-_i$, and $0$, otherwise.

In our experiments, we solve this model by \framework, using $K=10$ levels and the simple arc family
$\bfamily = \{ F_{(u,v)} = \{ (u,v) \} : (u,v) \in \arcset \}$,
which partitions $\arcset$ into individual arcs. 

\section{Computational Experiments} \label{sec:experiments}

This section discusses the results of our computational experiments.
The algorithms were coded in C++ and the LP and MILP models were solved by Gurobi 9.1.1. The experiments were run on a computer with an Intel Xeon E3-1245 v5 at 3.50GHz and 32GB RAM, with a single-thread limit.
First, based on the CSP, we evaluate the performance of the \framework{} components. Then, we also compare our results with those obtained by the state-of-the-art algorithms on each problem. 

\subsection{Experiments on the Cutting Stock Problem}

For the CSP, we consider the benchmark used to test the most recent exact methods, available at the BPPLIB \cite{DIM18}. They include classes Falkenauer, Hard, Scholl, Schwerin, and Waescher, which are all well-solved by state-of-the-art methods, and classes AI and ANI from \cite{DIM16}, which contain several open instances. Classes AI and ANI are both composed of 250 instances, divided into 5 groups of 50 instances having the same number of items. Instances in AI have optimal solution value equal to the optimal dual bound $z_{lb}$ of the pattern-based model, whereas instances in ANI have optimal solution value equal to $z_{lb}+1$.\\ 

\textbf{Comparison of variable-fixing strategies.}
We propose four algorithms to evaluate the RCVF effectiveness. They consist in applying (or not) the RCVF strategies from Section \ref{sec:variable_fixing} to some extent, and then solving the residual problem as an arc flow model by the MILP solver: in {\em No RCVF}, no variable-fixing is applied; {\em RCVF 1} uses the first strategy; {\em RCVF 1+2} uses the first and second strategies in sequence; and {\em RCVF 1+2+3} uses all strategies in sequence. A time limit of 600 seconds per instance is imposed. Table \ref{tab:csp_rcvf} gives average running times (time) in seconds, numbers of instances optimally solved (opt), and average percentages of arcs removed (rd. $\%$).

\begin{table}
\centering
\setlength{\tabcolsep}{2.5pt}
\caption{\small{Comparison of RCVF strategies for the CSP (time limit 600s)}}
\label{tab:csp_rcvf}
\small\begin{tabular}{lrrrrrrrrrrrr} 
\toprule
             & \multicolumn{1}{l}{}       & \multicolumn{2}{c}{No RCVF}                        & \multicolumn{3}{c}{RCVF 1}                                                       & \multicolumn{3}{c}{RCVF 1+2}                                                   & \multicolumn{3}{c}{RCVF 1+2+3}                                                \\
\cmidrule(lr){3-4} \cmidrule(lr){5-7} \cmidrule(lr){8-10} \cmidrule(lr){11-13}
Class        & \multicolumn{1}{c}{\# ins} & \multicolumn{1}{c}{time} & \multicolumn{1}{c}{opt} & \multicolumn{1}{c}{rd. \%} & \multicolumn{1}{c}{time} & \multicolumn{1}{c}{opt} & \multicolumn{1}{c}{rd. \%} & \multicolumn{1}{c}{time} & \multicolumn{1}{c}{opt} & \multicolumn{1}{c}{rd. \%} & \multicolumn{1}{c}{time} & \multicolumn{1}{c}{opt}  \\ 
\midrule
AI200        & 50                         & 29.9                     & \textbf{50}             & 35.2                        & 16.3                     & \textbf{50}             & 61.8                        & 14.4                     & \textbf{50}             & 66.6                        & 22.7                     & \textbf{50}              \\
AI400        & 50                         & 410.6                    & 27                      & 41.5                        & 242.8                    & 36                      & 78.5                        & 171.1                    & 38                      & 84.7                        & 141.8                    & \textbf{41}              \\
AI600        & 50                         & 600.0                    & 0                       & 30.8                        & 543.7                    & 7                       & 68.5                        & 350.4                    & 25                      & 78.3                        & 254.6                    & \textbf{33}              \\
AI800        & 50                         & 600.0                    & 0                       & 29.2                        & 593.5                    & 2                       & 74.0                        & 371.3                    & 24                      & 83.6                        & 301.4                    & \textbf{35}              \\
AI1000       & 50                         & 600.0                    & 0                       & 31.4                        & 589.7                    & 1                       & 66.0                        & 480.5                    & 14                      & 73.8                        & 433.3                    & \textbf{24}              \\
\cmidrule(lr){1-2} \cmidrule(lr){3-4} \cmidrule(lr){5-7} \cmidrule(lr){8-10} \cmidrule(lr){11-13}
ANI200       & 50                         & 53.9                     & \textbf{50}             & 61.1                        & 13.0                     & \textbf{50}             & 96.9                        & 2.0                      & \textbf{50}             & 97.9                        & 1.4                      & \textbf{50}              \\
ANI400       & 50                         & 546.0                    & 11                      & 47.0                        & 203.5                    & 39                      & 94.9                        & 47.7                     & 47                      & 99.7                        & 7.5                      & \textbf{50}              \\
ANI600       & 50                         & 600.0                    & 0                       & 29.4                        & 583.8                    & 4                       & 84.7                        & 204.4                    & 34                      & 99.8                        & 36.9                     & \textbf{50}              \\
ANI800       & 50                         & 600.0                    & 0                       & 31.6                        & 600.0                    & 0                       & 80.0                        & 328.0                    & 26                      & 98.0                        & 154.0                    & \textbf{47}              \\
ANI1000      & 50                         & 600.0                    & 0                       & 35.9                        & 600.0                    & 0                       & 79.7                        & 428.7                    & 20                      & 88.1                        & 357.9                    & \textbf{31}              \\
\cmidrule(lr){1-2} \cmidrule(lr){3-4} \cmidrule(lr){5-7} \cmidrule(lr){8-10} \cmidrule(lr){11-13}
Falkenauer T & 80                         & 0.3                      & \textbf{80}             & 4.4                         & 0.2                      & \textbf{80}             & 4.5                         & 0.3                      & \textbf{80}             & 8.7                         & 0.4                      & \textbf{80}              \\
Falkenauer U & 80                         & 0.2                      & \textbf{80}             & 3.2                         & 0.2                      & \textbf{80}             & 3.4                         & 0.3                      & \textbf{80}             & 3.4                         & 0.3                      & \textbf{80}              \\
Hard         & 28                         & 60.6                     & \textbf{28}             & 65.2                        & 30.0                     & \textbf{28}             & 69.5                        & 27.7                     & \textbf{28}             & 69.8                        & 28.9                     & \textbf{28}              \\
Random       & 3840                       & 1.9                      & \textbf{3838}           & 15.9                        & 2.2                      & \textbf{3838}           & 16.9                        & 2.3                      & \textbf{3838}           & 17.0                        & 15.0                     & 3823                     \\
Scholl       & 1210                       & 12.7                     & \textbf{1195}           & 12.9                        & 20.7                     & 1191                    & 13.8                        & 21.1                     & 1190                    & 13.8                        & 45.5                     & 1160                     \\
Schwerin     & 200                        & 3.5                      & \textbf{200}            & 3.5                         & 3.3                      & \textbf{200}            & 3.5                         & 3.3                      & \textbf{200}            & 3.6                         & 5.8                      & \textbf{200}             \\
Waescher     & 17                         & 298.0                    & \textbf{14}             & 17.4                        & 305.1                    & \textbf{14}             & 17.7                        & 306.1                    & \textbf{14}             & 17.7                        & 362.2                    & \textbf{14}              \\ 
\midrule
Overall      & 5955                       & 295.1                    & 5573                    & 29.1                        & 255.8                    & 5620                    & 53.8                        & 162.3                    & 5758                    & 59.1                        & 127.6                    & \textbf{5796}            \\
\bottomrule
\end{tabular}
\end{table}

Classes Falkenauer, Hard, and Schwerin are well-solved by all algorithms. The number of instances solved in class Waescher remains unchanged for all algorithms. Random and Scholl are the only classes in which applying the more expensive strategies worsens the number of instances solved. However, most of the instances in these classes fall in the case in which $z_{opt} = \lceil z_{lb} \rceil$, in which RCVF is not very helpful. However, by analyzing the subset of instances in these two classes in which $z_{opt} = \lceil z_{lb} \rceil + 1$, we noticed that RCVF led to a good average improvement.

For the hardest classes (AI and ANI), it is always worth to apply more expensive RCVF strategies, because the number of solved instances increases consistently. Already with the first strategy, a substantial reduction in the network size is observed. A larger improvement is obtained when the second and third strategies are used, especially for the ANI instances. These instances particularly benefit from RCVF, because at the beginning of all algorithms we already have the optimal solution (easily found by simple heuristics) and the challenge is just to prove optimality. Then, the use of sub-optimal dual solutions consistently strengthens the LP relaxation, and optimality of 134 out of 250 ANI instances is proven even before invoking the MILP solver. Clearly, this did not occurr in any AI instance. The average ratio of basic arcs removed per iteration for AI and ANI is, respectively, $10.1$ and $56.5$ in the second strategy, and $5.7$ and $8.4$ in the third strategy.
Overall, the results prove that applying the three strategies leads to better results on average.\\

%
\textbf{Analysis of the heuristic behavior of the arc families.} To analyze whether the branching scheme of Section \ref{sec:branching} is a good approach to quickly find optimal solutions, we solve a single left branch for six different arc families: $\bfamily^a_1$, $\bfamily^a_5$, $\bfamily^a_{10}$, $\bfamily^b_5$, $\bfamily^b_{10}$, and $\bfamily^b_{50}$. For each family, we solve the linear relaxation (no RCVF is applied) and a single left branch, to obtain a feasible solution. Then, optimality is determined by comparing the solution value with the dual bound. For the sake of conciseness we do not report extended results.

The best balance between number of instances solved to proven optimality and computational time is given by $\bfamily^a_1$.
With this family, the number of AI instances with $200$, $400$, $600$, $800$, and $1000$ items solved to proven optimality by this procedure is respectively $48$, $34$, $32$, $28$, and $16$, with average computational time being, respectively, $0.6$, $7.9$, $68.9$, $121.5$, and $237.9$ seconds. By comparing these results with column `No RCVF' in Table \ref{tab:csp_rcvf}, it is clear that the branching improves the MILP solver in quickly finding an optimal solution.
All next experiments on the CSP are based on arc family $\bfamily^a_1$.\\ 

%
\textbf{Comparison of different networks.}
To evaluate the effectiveness of the new waste-limited network, we provide a comparison with network based on the MIM patterns \cite{CI18}. For that, we solved the models derived from both networks by \framework, with a time limit of 600 seconds per instance. Table \ref{tab:csp_network} presents the results, giving average number of arcs (\# arcs), average time (time), and number of proven optimal solutions (opt).

In the AI and ANI classes, the waste-limited network performs better than the MIM, mainly because goal solutions (of value $\lceil z_{lb} \rceil$) of instances in these classes do not allow any waste, which greatly benefits the reduction based on Remark \ref{property:csp1}. 
The remaining classes are all well-solved by using both networks, except for  Waescher, in which three instances are always unsolved, and Scholl, in which the MIM performs better. Since there is no overall dominance among the networks, in the next experiments we compute both networks a priori and use smallest one.\\

\begin{table}
\centering
\caption{\small{Comparison of different networks (time limit 600s)}}
\small\begin{tabular}{lrrrrrrr}
\toprule
             & \multicolumn{1}{l}{}       & \multicolumn{3}{c}{MIM}                                                          & \multicolumn{3}{c}{Waste-limited network}                                                  \\

\cmidrule(lr){3-5} \cmidrule(lr){6-8}
Class        & \multicolumn{1}{l}{\# ins} & \multicolumn{1}{c}{\# arcs} & \multicolumn{1}{c}{time} & \multicolumn{1}{c}{opt} & \multicolumn{1}{c}{\# arcs} & \multicolumn{1}{c}{time} & \multicolumn{1}{c}{opt}  \\
\midrule
AI200        & 50                         & 85507.8                     & 10.4                     & \textbf{50}             & 59745.4                     & 2.0                      & \textbf{50}              \\
AI400        & 50                         & 671348.2                    & 169.8                    & 48                      & 507569.6                    & 25.2                     & \textbf{50}              \\
AI600        & 50                         & 2128423.7                   & 327.0                    & 37                      & 1670828.2                   & 118.1                    & \textbf{48}              \\
AI800        & 50                         & 6004603.6                   & 475.4                    & 23                      & 4789367.9                   & 271.2                    & \textbf{41}              \\
AI1000       & 50                         & 15134832.1                  & 596.8                    & 2                       & 11949055.6                  & 513.7                    & \textbf{18}              \\

\midrule

ANI200       & 50                         & 83947.6                     & 32.2                     & 49                      & 59289.7                     & 3.0                      & \textbf{50}              \\
ANI400       & 50                         & 668211.3                    & 278.1                    & 44                      & 503879.3                    & 24.9                     & \textbf{50}              \\
ANI600       & 50                         & 2121692.8                   & 487.5                    & 17                      & 1663330.5                   & 105.1                    & \textbf{49}              \\
ANI800       & 50                         & 5990194.7                   & 590.1                    & 4                       & 4772556.4                   & 291.8                    & \textbf{43}              \\
ANI1000      & 50                         & 15104071.5                  & 600.0                    & 0                       & 11915143.1                  & 507.9                    & \textbf{21}              \\

\midrule

Falkenauer T & 80                         & 4840.1                      & 0.3                      & \textbf{80}             & 1285.8                      & 0.3                      & \textbf{80}              \\
Falkenauer U & 80                         & 1513.8                      & 0.1                      & \textbf{80}             & 2906.6                      & 0.1                      & \textbf{80}              \\
Hard         & 28                         & 22219.4                     & 23.7                     & \textbf{27}             & 27066.5                     & 24.8                     & \textbf{27}              \\
Random       & 3840                       & 4238.4                      & 0.6                      & \textbf{3840}           & 9833.7                      & 0.9                      & \textbf{3840}            \\
Scholl       & 1210                       & 11384.0                     & 1.3                      & \textbf{1210}           & 26426.9                     & 12.9                     & 1193                     \\
Schwerin     & 200                        & 5560.4                      & 0.2                      & \textbf{200}            & 11928.4                     & 1.0                      & \textbf{200}             \\
Waescher      & 17                         & 101141.2                    & 211.9                    & \textbf{14}             & 121536.2                    & 334.9                    & \textbf{14}              \\
\midrule
Overall      & 5955                       & 2831984.1                   & 223.8                    & 5725                    & 2240691.2                   & 131.6                    & \textbf{5854}           \\
\bottomrule
\end{tabular}
\label{tab:csp_network}
\end{table}

\textbf{Comparison with the state-of-the-art.}
We compare the results obtained by \framework{} with: the enhanced solution of the reflect formulation (a pseudo-polynomial CSP arc flow model) by Delorme and Iori \cite{DI20}, solved with an Intel Xeon at 3.10 GHz and 8 GB RAM; the branch-and-cut-and-price algorithm by Wei et al. \cite{WLBL20}, solved with an Intel Xeon E5-1603 at 2.80-GHz and 8 GB RAM; and the branch-and-cut-and-price algorithm for general network flow problems with resource constraints by Pessoa et al. \cite{PSUV20}, solved with an Intel Xeon E5-2680 at 2.50 GHz with 128 GB RAM shared by 8 copies of the algorithm running in parallel. All experiments considered a time limit of 3600 seconds per instance. To compare the performance of each CPU, we provide their single-thread passmark indicators (STPI) (available at \url{www.passmark.com}), where higher values are associated with better performance. The computer used in the experiments of \cite{DI20}, \cite{PSUV20}, \cite{WLBL20}, and the present work have STPI $2132$, $1763$, $1635$, and $1739$, respectively.

\begin{table}
\centering
\caption{\small{Comparison with the state-of-the-art for the CSP (time limit 3600s)}}
\label{tab:csp_state_of_the_art}
\setlength{\tabcolsep}{3.5pt}
\small\begin{tabular}{lrrrrrrrrrrr} 
\toprule
                                  & \multicolumn{1}{l}{}       & \multicolumn{2}{c}{DI \cite{DI20}}                             & \multicolumn{2}{c}{WLBL \cite{WLBL20}}                           & \multicolumn{2}{c}{PSUV \cite{PSUV20}}                           & \multicolumn{2}{c}{Arc Flow}                       & \multicolumn{2}{c}{\framework}                       \\
\cmidrule(lr){3-4} \cmidrule(lr){5-6} \cmidrule(lr){7-8} \cmidrule(lr){9-10} \cmidrule(lr){11-12} 
Class & \multicolumn{1}{c}{\# ins} & \multicolumn{1}{c}{time} & \multicolumn{1}{c}{opt} & \multicolumn{1}{c}{time} & \multicolumn{1}{c}{opt} & \multicolumn{1}{c}{time} & \multicolumn{1}{c}{opt} & \multicolumn{1}{c}{time} & \multicolumn{1}{c}{opt} & \multicolumn{1}{c}{time} & \multicolumn{1}{c}{opt}  \\
\midrule
AI200                             & 50                         & 8.5                      & \textbf{50}             & 4.2                      & \textbf{50}             & 52.3                     & \textbf{50}             & 21.5                     & \textbf{50}             & 2.0                      & \textbf{50}              \\
AI400                             & 50                         & 1205.0                   & 40                      & 398.1                    & 46                      & 491.4                    & 47                      & 904.2                    & 44                      & 25.2                     & \textbf{50}              \\
AI600                             & 50                         & \multicolumn{1}{c}{-}    & \multicolumn{1}{c}{-}   & 1759.6                   & 27                      & 1454.1                   & 35                      & 3326.9                  & 9                       & 192.4                    & \textbf{49}              \\
AI800                             & 50                         & \multicolumn{1}{c}{-}    & \multicolumn{1}{c}{-}   & 2766.3                   & 15                      & 2804.7                   & 28                      & 3600.0                   & 0                       & 566.5                    & \textbf{46}              \\
AI1000                            & 50                         & \multicolumn{1}{c}{-}    & \multicolumn{1}{c}{-}   & 3546.1                   & 2                       & \multicolumn{1}{c}{-}    & \multicolumn{1}{c}{-}   & 3600.0                   & 0                       & 1577.1                   & \textbf{36}              \\

\midrule
ANI200                            & 50                         & 49.3                     & \textbf{50}             & 13.9                     & \textbf{50}             & 16.7                     & \textbf{50}             & 16.3                     & \textbf{50}             & 3.0                      & \textbf{50}              \\
ANI400                            & 50                         & 2703.9                   & 17                      & 436.2                    & 47                      & 96.0                     & \textbf{50}             & 1252.5                   & 42                      & 24.9                     & \textbf{50}              \\
ANI600                            & 50                         & \multicolumn{1}{c}{-}    & \multicolumn{1}{c}{-}   & 3602.7                   & 0                       & 3512.5                   & 3                       & 3473.5                  & 6                       & 140.7                    & \textbf{50}              \\
ANI800                            & 50                         & \multicolumn{1}{c}{-}    & \multicolumn{1}{c}{-}   & 3605.9                   & 0                       & 3600.0                   & 0                       & 3600.0                   & 0                       & 393.2                    & \textbf{49}              \\
ANI1000                           & 50                         & \multicolumn{1}{c}{-}    & \multicolumn{1}{c}{-}   & 3637.7                   & 0                       & \multicolumn{1}{c}{-}    & \multicolumn{1}{c}{-}   & 3600.0                   & 0                       & 1302.5                   & \textbf{43}              \\

\midrule
Falkenauer T                      & 80                         & 1.0                      & \textbf{80}             & 1.9                      & \textbf{80}             & 16.0                     & \textbf{80}             & 0.8                        & \textbf{80}                       & 0.3                      & \textbf{80}              \\
Falkenauer U                      & 80                         & 0.1                      & \textbf{80}             & 3.8                      & \textbf{80}             & \multicolumn{1}{c}{-}    & \multicolumn{1}{c}{-}   & 0.1                        & \textbf{80}                       & 0.1                      & \textbf{80}              \\
Hard                              & 28                         & 4.2                      & \textbf{28}             & 41.5                     & \textbf{28}             & 17.0                     & \textbf{28}             & 39.9                        &\textbf{28}                       & 23.6                     & \textbf{28}              \\
Random                            & 3840                       & \multicolumn{1}{c}{-}    & \multicolumn{1}{c}{-}   & 6.2                      & \textbf{3840}           & \multicolumn{1}{c}{-}    & \multicolumn{1}{c}{-}   & 1.4                        & \textbf{3840}                       & 0.9                      & \textbf{3840}            \\
Scholl                            & 1210                       & 6.6                      & \textbf{1210}           & 5.0                      & \textbf{1210}           & \multicolumn{1}{c}{-}    & \multicolumn{1}{c}{-}   & 8.2                        & \textbf{1210}                       & 1.4                      & \textbf{1210}            \\
Schwerin                          & 200                        & 0.2                      & \textbf{200}            & 0.3                      & \textbf{200}            & \multicolumn{1}{c}{-}    & \multicolumn{1}{c}{-}   & 1.3                        & \textbf{200}            & 0.2                      & \textbf{200}             \\
Waescher                          & 17                         & 41.3                     & \textbf{17}             & 8.7                      & \textbf{17}             & \multicolumn{1}{c}{-}    & \multicolumn{1}{c}{-}   & 1510.0                        & \textbf{17}                       & 161.2                    & \textbf{17}              \\
\bottomrule
\end{tabular}
\end{table}

The results are presented in Table \ref{tab:csp_state_of_the_art}. Under columns `Arc Flow' and `\framework', we report the results of the arc flow model solved, respectively, directly by Gurobi and by \framework.
For the AI and ANI classes, the best previous results were obtained by Pessoa et al. \cite{PSUV20}, who optimally solved 263 out of 500 instances. \framework{} raised the number of proven optima to 430, and it could solve for the first time many ANI instances with 800 and 1000 items, mainly due to the new RCVF strategies. The LP relaxation strengthening obtained by these strategies allowed some of these instances to be solved already at the root node. The new branching scheme helped finding an optimal solution for several hard AI instances within a reasonably short computational time. For the remaining classes, the method by Wei et al. \cite{WLBL20} already performed well. Our method could improve the average solution time of all classes with the exception of Waescher, because of a few instances that required a large computational time in the solution of specific left branches.

\subsection{Experiments on the Two-Staged Cutting Stock Problem}
For the 2GCSP, we compare \framework{} with the most recent exact methods for the problem, namely: the B\&P by Mrad et al. \cite{MMH13}; and the pseudo-polynomial MILP models by Macedo et al. \cite{MAV10}, Silva et al. \cite{SAV10}, and C\^ot\'e and Iori \cite{CI18}. The best published results are the ones by the arc flow model in \cite{CI18}, which is also the model that we solve with \framework.
The experiments were based on benchmark classes A and APT, and their respective variants A-r and APT-r in which items and stock sheets are rotated by 90 degrees. These instances can be downloaded from the 2DPackLib (\url{http://or.dei.unibo.it/library/2dpacklib}).

Table \ref{tab:2gcsp_state_of_the_art} reports, for each algorithm and each class, the average time and the number of instances solved to proven optimality.
The algorithm in \cite{MMH13} was solved on a Pentium IV at 2.2 GHz with 4 GB RAM (STPI 562), and the models in \cite{MAV10,SAV10} were solved on an Intel Core Duo at 1.87 GHz with 2 GB RAM (STPI 675). We also compare \framework{} with the arc flow model solved directly by Gurobi, which is the approach in \cite{CI18}. Thus, column CI does not report the results in \cite{CI18}, but rather the updated (and improved) results obtained by solving their model by the current version of Gurobi in our computer. All results consider a time limit of 7200 seconds per instance.

Although not all previous publications addressed all classes, we can see that the arc flow model in \cite{CI18} dominates on average the previous works. \framework{} provides the best results overall: It uses a smaller average time, mainly due to the effectiveness of the branching scheme, and solves all instances in class A-r for the first time. The six remaining open instances in classes APT and APT-r have an absolute gap of just one stock sheet.

\begin{table}
\centering
\setlength{\tabcolsep}{4.5pt}
\caption{\small{Comparison with the state-of-the-art for the 2GCSP (time limit 7200s)}}
\label{tab:2gcsp_state_of_the_art}
\small\begin{tabular}{lrrrccrrrrrr}
\toprule
        &        & \multicolumn{2}{c}{MMH \cite{MMH13}}                            & \multicolumn{2}{c}{MAV \cite{MAV10}}                                     & \multicolumn{2}{c}{SAV \cite{SAV10}}                            & \multicolumn{2}{c}{CI \cite{CI18}$^*$}                            & \multicolumn{2}{c}{\framework}                      \\
\cmidrule(lr){3-4} \cmidrule(lr){5-6} \cmidrule(lr){7-8} \cmidrule(lr){9-10} \cmidrule(lr){11-12} 
Class   & \# ins & \multicolumn{1}{c}{time} & \multicolumn{1}{c}{opt} & \multicolumn{1}{c}{time}  & \multicolumn{1}{c}{opt}         & \multicolumn{1}{c}{time} & \multicolumn{1}{c}{opt} & \multicolumn{1}{c}{time} & \multicolumn{1}{c}{opt} & \multicolumn{1}{c}{time} & \multicolumn{1}{c}{opt}  \\

\midrule

A   & 43     & 277.2                    & 42                      & \multicolumn{1}{r}{377.9} & \multicolumn{1}{r}{\textbf{43}} & 735                      & 41                      & 1.9                      & \textbf{43}             & 0.3                      & \textbf{43}              \\
A-r   & 43     & 1014.2                   & 37                      & -                         & -                               & \multicolumn{1}{c}{-}    & \multicolumn{1}{c}{-}   & 606.9                    & 40                      & 59.7                     & \textbf{43}              \\
APT & 20     & \multicolumn{1}{c}{-}    & \multicolumn{1}{c}{-}   & -                         & -                               & 3642.6                   & 12                      & 729.6                    & \textbf{18}             & 472.4                    & \textbf{18}              \\
APT-r & 20     & \multicolumn{1}{c}{-}    & \multicolumn{1}{c}{-}   & -                         & -                               & \multicolumn{1}{c}{-}    & \multicolumn{1}{c}{-}   & 1599.4                   & \textbf{16}             & 1557.4                   & \textbf{16}             \\
\bottomrule
\multicolumn{12}{l}{$*$: results of our reimplementation of the arc flow model in \cite{CI18}}  
\end{tabular}
\end{table}

\subsection{Experiments on the Skiving Stock Problem} \label{sec:experiments_ssp}

\begin{table}[b]
\centering
\caption{\small{Comparison with the state-of-the-art for the SSP (time limit 3600s)}}
\label{tab:ssp_state_of_the_art}
\small\begin{tabular}{crrrrrrr}
\toprule
\multicolumn{1}{l}{}        & \multicolumn{1}{l}{}       & \multicolumn{2}{c}{MDISS \cite{MDISS20}}                        & \multicolumn{2}{c}{MS \cite{MS16a}$^*$}                         & \multicolumn{2}{c}{\framework}                      \\
\cmidrule(lr){3-4} \cmidrule(lr){5-6} \cmidrule(lr){7-8}
\multicolumn{1}{l}{Class}   & \multicolumn{1}{l}{\# ins} & \multicolumn{1}{c}{time} & \multicolumn{1}{c}{opt} & \multicolumn{1}{c}{time} & \multicolumn{1}{c}{opt} & \multicolumn{1}{c}{time} & \multicolumn{1}{c}{opt}  \\
\midrule
A1                           & 1260                       & 0.1                      & \textbf{1260}           & 0.1                      & \textbf{1260}           & 0.1                      & \textbf{1260}            \\
A2          & 1050                       & 250.9                    & 1011                    & 183.4                    & \textbf{1030}           & 148.7                    & 1024                     \\

B  & 160                        & 970.7                    & 126                     & 1354.5                   & 136                     & 61.4                     & \textbf{160}            \\
\bottomrule
\multicolumn{8}{l}{$*$: results of our reimplementation of the arc flow model in \cite{MS16a}}            
\end{tabular}
\end{table}
For the SSP, we compare \framework{} with the reflect formulation in \cite{MDISS20} and with a reimplementation of the arc flow model in \cite{MS16a} executed with the latest version of Gurobi on our computer.
The results in \cite{MDISS20} were obtained by an AMD A10-5800K with 16 GB RAM (STPI 1493).
Our experiments are based on the three benchmark classes A1, composed of 1260 easy instances, A2, composed of 1050 hard instances, and B, composed of 160 hard instances. The arc flow model at the basis of \framework{} is the one in \cite{MS16a}.

Table \ref{tab:ssp_state_of_the_art} gives average times and numbers of instances optimally solved. For class A1, all methods could solve all instances very quickly. For class A2, all methods could solve all instances with up to 100 items, but \framework{} was quicker. For instances with 250 and 500 items, \framework{} could not improve the solution of the arc flow model as a (general) MILP. The reason is that for some instances there are left branches whose resulting problems, although consistently smaller, are very hard to solve. We believe that the source of such difficulties is often an excessive restriction on the number of optimal solutions remaining in these branches. For class B, \framework{} improved upon the other methods and optimally solved all instances for the first time. It consistently decreased the average time w.r.t. the direct solution of the arc flow as a MILP, showing that the branching scheme is very effective for all instances in this class.

\subsection{Experiments on the Ordered Open-End Bin Packing Problem} \label{sec:experiments_ooebpp}
For the OOEBPP, we solve the arc flow model derived from the set covering formulation in \cite{CR08}, either directly by Gurobi or by \framework{}. We compare our results also with the ones obtained by the B\&P in \cite{CR08}, which were obtained by a Pentium IV 1.6 GHz with 512 MB of RAM (STPI 562). 

The first test is based on the benchmark used in \cite{CR08}, which are available at the 2DPackLib.
The results are reported in Table \ref{tab:ooebpp_state_of_the_art}. Ceselli and Righini \cite{CR08} did not solve instances GCUT5-13. Their B\&P already performed very well and left just one open instance. The arc flow model solved either directly by Gurobi or by \framework{} could close all instances (the latter approach used a slightly smaller time).
Since the available benchmarks for the OOEBPP appear to be very easy for current solvers, we propose a second set of experiments based on new randomly generated instances that we derived from the CSP class Random. We consider all instances with 50, 100, and 200 items. The ordering of the items required by the OOEBPP is obtained by sorting the original items by weights randomly generated from a uniform distribution.
The results for these instances are also presented in Table \ref{tab:ooebpp_state_of_the_art}. Overall, \framework{} performs better than arc flow both in terms of average time and number of instances solved. The arc flow solution as a MILP left 55 open instances, but with \framework{} this number decreased to 15, while using a consistently smaller computing time. 

\begin{table}[tbp]
\centering
\caption{\small{Comparison with the state-of-the-art for the OOEBPP (time limit 3600s)}}
\label{tab:ooebpp_state_of_the_art}
\small\begin{tabular}{lrrrrrrr}
\toprule
         & \multicolumn{1}{l}{}       & \multicolumn{2}{c}{CR \cite{CR08}}    & \multicolumn{2}{c}{Arc Flow}                         & \multicolumn{2}{c}{\framework}                   \\
\cmidrule(lr){3-4} \cmidrule(lr){5-6} \cmidrule(lr){7-8}
Class    & \multicolumn{1}{c}{\# ins} & \multicolumn{1}{c}{time} & \multicolumn{1}{c}{opt} & \multicolumn{1}{c}{time} & \multicolumn{1}{c}{opt} & \multicolumn{1}{c}{time} & \multicolumn{1}{c}{opt}  \\
\midrule
GCUT1-4  & 4                          & 0.6                     & \textbf{4}              & 0.1                      & \textbf{4}              & 0.0                      & \textbf{4}               \\
GCUT5-13 & 9                          & \multicolumn{1}{c}{-}    & \multicolumn{1}{c}{-}   & 1.3                      & \textbf{9}              & 0.1                      & \textbf{9}               \\
NGCUT    & 12                         & 0.1                     & \textbf{12}             & 0.0                      & \textbf{12}             & 0.0                      & \textbf{12}              \\
CGCUT    & 3                          & 0.1                     & \textbf{3}              & 0.3                      & \textbf{3}              & 0.4                      & \textbf{3}               \\
BENG     & 10                         & 0.1                     & \textbf{10}             & 0.3                      & \textbf{10}             & 0.1                      & \textbf{10}              \\
HT       & 9                          & 0.1                     & \textbf{9}              & 0.0                      & \textbf{9}              & 0.0                      & \textbf{9}               \\
CLASS    & 500                        & 8.7                      & 499                     & 11.3                     & \textbf{500}            & 3.7                      & \textbf{500}             \\
\midrule
Overall  & 547                        & 1.6                      & 537                     & 1.9                      & \textbf{547}            & 0.6                      & \textbf{547}            \\
\midrule

Random50  & \multicolumn{1}{r}{480} & \multicolumn{1}{c}{-}& \multicolumn{1}{c}{-}& 1.7                      & \textbf{480}            & 1.3                      & \textbf{480}             \\
Random100 & \multicolumn{1}{r}{480} & \multicolumn{1}{c}{-}& \multicolumn{1}{c}{-}& 73.4                     & \textbf{479}            & 38.7                     & \textbf{479}             \\
Random200 & \multicolumn{1}{r}{480} & \multicolumn{1}{c}{-}& \multicolumn{1}{c}{-}& 863.3                    & 426                     & 168.9                    & \textbf{466}             \\
\midrule
Overall                 &   \multicolumn{1}{r}{1440} & \multicolumn{1}{c}{-} & \multicolumn{1}{c}{-}                     & 312.8                    & 1385                    & 69.6                     & \textbf{1425}           \\
\bottomrule
\end{tabular}
\end{table}

\section{Conclusions} \label{sec:conclusions}

We proposed a framework for the exact solution of network flow models. It is tailored to efficiently solve pseudo-polynomial arc flow models that have very strong relaxation but also a huge number of arcs. Its main practical components include a general column(-and-row) generation algorithm for network flow models, RCVF strategies that explore alternative and possibly sub-optimal dual solutions, and a highly asymmetric branching scheme that exploits the potential of MILP solvers.
The correctness of the LP-based methods and of the RCVF strategies are supported by a number of theoretical results.

We performed extensive computational experiments on well-studied C\&P problems, in which we solved a large number of instances to proven optimality for the first time. The new RCVF strategies closed many hard instances already at the root node. 
The new branching scheme helped finding optimal solutions for many hard instances within short computational time, largely improving the results by state-of-the-art algorithms on all problems addressed.

Although we focus on network flow, an extension to more general DWM characterizations is also envisaged. 
The core of the proposed RCVF strategies and branching scheme lies in the ability to efficiently forbid pricing decisions (which in our case are given as arcs in a network). By having this ability in a general pricing algorithm, one can extend such techniques to solve a DWM without relying on a network flow characterization.
In this case, the restricted problem given in the left branch would be solved by alternative methods (as, e.g., specialized combinatorial algorithms) instead than as an arc flow model. 

Other than the C\&P problems presented in this paper, we also tested the solution of a number of other problems by our framework in preliminary experiments. This revealed some interesting drawbacks that we would like to address in future research:
\begin{enumerate}[label=(\roman*)]
\item Iterative aggregation/disaggregation is a state-of-the-art technique to deal with the huge number of arcs in pseudo-polynomial arc flow models for a couple of problems, as, e.g., the time-dependent traveling salesman problem with time windows (see \cite{VHBS19}). In some of these problems, hard instances often produce networks so huge that cannot even fit in the computer memory. This is a critical issue for our framework, because the full network is required in input. We believe that embedding aggregation/disaggregation techniques within the framework is an interesting research direction that may improve the solution of such problems;
\item Our branching scheme is tailored to be effective on models with very strong relaxations, but we tested models with weaker relaxations derived, for instance, from vehicle routing and scheduling problems. As expected, the results were not competitive with the state-of-the-art. The left branch generates small arc flow models, but their weak linear relaxations usually does not allow an efficient solution by a general MILP solver. In addition, the number of right-branch constraints required to significantly raise the bound is often unpractical. Then, to better address such models, we aim at investigating effective primal cuts tailored for general arc flow models.
\end{enumerate}
\subsection*{Acknowledgments}
We thank three anonymous referees of a preliminary version of this paper \cite{LIM21} for their constructive feedback. We also thank Jos\'e M. Val{\'e}rio de Carvalho for the meaningful discussions on the linear relaxation of network flow models, and Fabio Furini for the discussions on the safety of dual solutions in column generation. Finally, we acknowledge the support by CNPq (Proc.~314366/2018-0, 425340/2016-3) and by FAPESP (Proc.~2015/11937-9, 2016/01860-1, 2017/11831-1).

\bibliographystyle{plainnat}

\begin{thebibliography}{43}
\providecommand{\natexlab}[1]{#1}
\providecommand{\url}[1]{\texttt{#1}}
\expandafter\ifx\csname urlstyle\endcsname\relax
  \providecommand{\doi}[1]{doi: #1}\else
  \providecommand{\doi}{doi: \begingroup \urlstyle{rm}\Url}\fi

\bibitem[Ahuja et~al.(1993)Ahuja, Magnanti, and Orlin]{AMO93}
R.K. Ahuja, T.L. Magnanti, and J.B. Orlin.
\newblock \emph{Network flows: theory, algorithms, and applications}.
\newblock Prentice-Hall, 1993.

\bibitem[Alves and {Val{\'{e}}rio de Carvalho}(2008)]{AV08}
C.~Alves and J.M. {Val{\'{e}}rio de Carvalho}.
\newblock A stabilized branch-and-price-and-cut algorithm for the multiple
  length cutting stock problem.
\newblock \emph{Computers \& Operations Research}, 35\penalty0 (4):\penalty0
  1315--1328, 2008.

\bibitem[Bajgiran et~al.(2017)Bajgiran, Cire, and Rousseau]{BCR17}
O.S. Bajgiran, A.A. Cire, and L.-M. Rousseau.
\newblock A first look at picking dual variables for maximizing reduced cost
  fixing.
\newblock In D.~Salvagnin and M.~Lombardi, editors, \emph{Integration of AI and
  OR Techniques in Constraint Programming}, pages 221--228. Springer
  International Publishing, 2017.

\bibitem[Bergman et~al.(2015)Bergman, Cire, and {van Hoeve}]{BCH15}
D.~Bergman, A.A. Cire, and W.J. {van Hoeve}.
\newblock Lagrangian bounds from decision diagrams.
\newblock \emph{Constraints}, 20:\penalty0 346--361, 2015.

\bibitem[Caprara et~al.(2015)Caprara, Dell'Amico, D{\'i}az-D{\'i}az, Iori, and
  Rizzi]{CDDIR15}
A.~Caprara, M.~Dell'Amico, J.C. D{\'i}az-D{\'i}az, M.~Iori, and R.~Rizzi.
\newblock Friendly bin packing instances without integer round-up property.
\newblock \emph{Mathematical Programming}, 150:\penalty0 5--17, 2015.

\bibitem[Ceselli and Righini(2008)]{CR08}
A.~Ceselli and G.~Righini.
\newblock An optimization algorithm for the ordered open-end bin-packing
  problem.
\newblock \emph{Operations Research}, 56\penalty0 (2):\penalty0 425--436, 2008.

\bibitem[Christofides et~al.(1981)Christofides, Mingozzi, and Toth]{CMT81}
N.~Christofides, A.~Mingozzi, and P.~Toth.
\newblock State-space relaxation procedures for the computation of bounds to
  routing problems.
\newblock \emph{Networks}, 11\penalty0 (2):\penalty0 145--164, 1981.

\bibitem[C{\^o}t{\'e} and Iori(2018)]{CI18}
J.-F. C{\^o}t{\'e} and M.~Iori.
\newblock The meet-in-the-middle principle for cutting and packing problems.
\newblock \emph{INFORMS Journal on Computing}, 30\penalty0 (4):\penalty0
  646--661, 2018.

\bibitem[Dantzig and Wolfe(1961)]{DW61}
G.B. Dantzig and P.~Wolfe.
\newblock The decomposition algorithm for linear programs.
\newblock \emph{Econometrica}, 29\penalty0 (4):\penalty0 767--778, 1961.

\bibitem[{de Lima} et~al.(2021){de Lima}, Iori, and Miyazawa]{LIM21}
V.L. {de Lima}, M.~Iori, and F.K. Miyazawa.
\newblock New exact techniques applied to a class of network flow formulations.
\newblock In Mohit Singh and David~P. Williamson, editors, \emph{Integer
  Programming and Combinatorial Optimization}, pages 178--192, Cham, 2021.
  Springer.

\bibitem[{de Lima} et~al.(2021, forthcoming){de Lima}, Alves, Clautiaux, Iori,
  and {Val{\'e}rio de Carvalho}]{LACIV21}
V.L. {de Lima}, C.~Alves, F.~Clautiaux, M.~Iori, and J.M. {Val{\'e}rio de
  Carvalho}.
\newblock Arc flow formulations based on dynamic programming: Theoretical
  foundations and applications.
\newblock \emph{European Journal of Operational Research}, 2021, forthcoming.

\bibitem[Delorme and Iori(2020)]{DI20}
M.~Delorme and M.~Iori.
\newblock Enhanced pseudo-polynomial formulations for bin packing and cutting
  stock problems.
\newblock \emph{INFORMS Journal on Computing}, 32\penalty0 (1):\penalty0
  101--119, 2020.

\bibitem[Delorme et~al.(2016)Delorme, Iori, and Martello]{DIM16}
M.~Delorme, M.~Iori, and S.~Martello.
\newblock Bin packing and cutting stock problems: Mathematical models and exact
  algorithms.
\newblock \emph{European Journal of Operational Research}, 255\penalty0
  (1):\penalty0 1--20, 2016.

\bibitem[Delorme et~al.(2018)Delorme, Iori, and Martello]{DIM18}
M.~Delorme, M.~Iori, and S.~Martello.
\newblock {BPPLIB:} a library for bin packing and cutting stock problems.
\newblock \emph{Optimization Letters}, 12\penalty0 (2):\penalty0 235--250,
  2018.

\bibitem[Desaulniers et~al.(2006)Desaulniers, Desrosiers, and Solomon]{DDS06}
G.~Desaulniers, J.~Desrosiers, and M.M. Solomon.
\newblock \emph{Column Generation}.
\newblock Springer Science \& Business Media, 2006.

\bibitem[Fischetti and Lodi(2003)]{FL03}
M.~Fischetti and A.~Lodi.
\newblock Local branching.
\newblock \emph{Mathematical Programming}, 98:\penalty0 23--47, 2003.

\bibitem[Fukasawa et~al.(2006)Fukasawa, Longo, Lysgaard, {De Arag{\~a}o}, Reis,
  Uchoa, and Werneck]{FLLARUW06}
R.~Fukasawa, H.~Longo, J.~Lysgaard, M.P. {De Arag{\~a}o}, M.~Reis, E.~Uchoa,
  and R.F. Werneck.
\newblock Robust branch-and-cut-and-price for the capacitated vehicle routing
  problem.
\newblock \emph{Mathematical programming}, 106\penalty0 (3):\penalty0 491--511,
  2006.

\bibitem[Gilmore and Gomory(1961)]{GG61}
P.C. Gilmore and R.E. Gomory.
\newblock A linear programming approach to the cutting-stock problem.
\newblock \emph{Operations Research}, 9\penalty0 (6):\penalty0 849--859, 1961.

\bibitem[Gilmore and Gomory(1963)]{GG63}
P.C. Gilmore and R.E. Gomory.
\newblock A linear programming approach to the cutting stock problem - part
  {II}.
\newblock \emph{Operations Research}, 11\penalty0 (6):\penalty0 863--888, 1963.

\bibitem[Gleixner et~al.(2017)Gleixner, Berthold, M{\"u}ller, and
  Weltge]{GBMW17}
A.M. Gleixner, T.~Berthold, B.~M{\"u}ller, and S.~Weltge.
\newblock Three enhancements for optimization-based bound tightening.
\newblock \emph{Journal of Global Optimization}, 67\penalty0 (4):\penalty0
  731--757, 2017.

\bibitem[Hadjar et~al.(2006)Hadjar, Marcotte, and Soumis]{HMS06}
A.~Hadjar, O.~Marcotte, and F.~Soumis.
\newblock A branch-and-cut algorithm for the multiple depot vehicle scheduling
  problem.
\newblock \emph{Operations Research}, 54\penalty0 (1):\penalty0 130--149, 2006.

\bibitem[Held et~al.(2012)Held, Cook, and Sewell]{HCS12}
S.~Held, W.~Cook, and E.C. Sewell.
\newblock Maximum-weight stable sets and safe lower bounds for graph coloring.
\newblock \emph{Mathematical Programming Computation}, 4:\penalty0 363--381,
  2012.

\bibitem[Irnich et~al.(2010)Irnich, Desaulniers, Desrosiers, and
  Hadjar]{IDDH10}
S.~Irnich, G.~Desaulniers, J.~Desrosiers, and A.~Hadjar.
\newblock Path-reduced costs for eliminating arcs in routing and scheduling.
\newblock \emph{INFORMS Journal on Computing}, 22\penalty0 (2):\penalty0
  297--313, 2010.

\bibitem[L{\"u}bbecke and Desrosiers(2005)]{LD05}
M.E. L{\"u}bbecke and J.~Desrosiers.
\newblock Selected topics in column generation.
\newblock \emph{Operations Research}, 53\penalty0 (6):\penalty0 1007--1023,
  2005.

\bibitem[Macedo et~al.(2010)Macedo, Alves, and {Val\'{e}rio de
  Carvalho}]{MAV10}
R.~Macedo, C.~Alves, and J.M. {Val\'{e}rio de Carvalho}.
\newblock Arc-flow model for the two-dimensional guillotine cutting stock
  problem.
\newblock \emph{Computers \& Operations Research}, 37\penalty0 (6):\penalty0
  991--1001, 2010.

\bibitem[Martello and Toth(1990)]{MT90}
S.~Martello and P.~Toth.
\newblock \emph{Knapsack Problems: Algorithms and Computer Implementations}.
\newblock John Wiley \& Sons, Inc., 1990.

\bibitem[Martinovic and Scheithauer(2016)]{MS16a}
J.~Martinovic and G.~Scheithauer.
\newblock Integer linear programming models for the skiving stock problem.
\newblock \emph{European Journal of Operational Research}, 251\penalty0
  (2):\penalty0 356--368, 2016.

\bibitem[Martinovic et~al.(2020, forthcoming)Martinovic, Delorme, Iori,
  Scheithauer, and Strasdat]{MDISS20}
J.~Martinovic, M.~Delorme, M.~Iori, G.~Scheithauer, and N.~Strasdat.
\newblock Improved flow-based formulations for the skiving stock problem.
\newblock \emph{Computers \& Operations Research}, 113, 2020, forthcoming.

\bibitem[Mrad et~al.(2013)Mrad, Meftahi, and Haouari]{MMH13}
M.~Mrad, I.~Meftahi, and M.~Haouari.
\newblock A branch-and-price algorithm for the two-stage guillotine cutting
  stock problem.
\newblock \emph{Journal of the Operational Research Society}, 64\penalty0
  (5):\penalty0 629--637, 2013.

\bibitem[Nemhauser and Wolsey(1988)]{NW88}
G.~Nemhauser and L.A. Wolsey.
\newblock \emph{Integer and Combinatorial Optimization}.
\newblock Wiley \& Sons, 1988.

\bibitem[Pessoa et~al.(2008)Pessoa, {de Arag{\~a}o}, and Uchoa]{PPU08}
A.~Pessoa, M.P. {de Arag{\~a}o}, and E.~Uchoa.
\newblock Robust branch-cut-and-price algorithms for vehicle routing problems.
\newblock In B.~Golden, S.~Raghavan, and E.~Wasil, editors, \emph{The Vehicle
  Routing Problem: Latest Advances and New Challenges}, pages 297--325.
  Springer US, 2008.

\bibitem[Pessoa et~al.(2010)Pessoa, Uchoa, {de Arag{\~a}o}, and
  Rodrigues]{PUPR10}
A.~Pessoa, E.~Uchoa, M.P. {de Arag{\~a}o}, and R.~Rodrigues.
\newblock Exact algorithm over an arc-time-indexed formulation for parallel
  machine scheduling problems.
\newblock \emph{Mathematical Programming Computation}, 2:\penalty0 259--290,
  2010.

\bibitem[Pessoa et~al.(2020)Pessoa, Sadykov, Uchoa, and Vanderbeck]{PSUV20}
A.~Pessoa, R.~Sadykov, E.~Uchoa, and F.~Vanderbeck.
\newblock A generic exact solver for vehicle routing and related problems.
\newblock \emph{Mathematical Programming}, 183:\penalty0 483--523, 2020.

\bibitem[Sadykov and Vanderbeck(2013)]{SV13}
R.~Sadykov and F.~Vanderbeck.
\newblock Column generation for extended formulations.
\newblock \emph{{EURO} Journal on Computational Optimization}, 1:\penalty0
  81--115, 2013.

\bibitem[Sellmann(2004)]{S04}
M.~Sellmann.
\newblock Theoretical foundations of cp-based lagrangian relaxation.
\newblock In M.~Wallace, editor, \emph{Principles and Practice of Constraint
  Programming}, pages 634--647, Berlin, Heidelberg, 2004. Springer Berlin
  Heidelberg.

\bibitem[Silva et~al.(2010)Silva, Alvelos, and {Val\'{e}rio de
  Carvalho}]{SAV10}
E.~Silva, F.~Alvelos, and J.M. {Val\'{e}rio de Carvalho}.
\newblock An integer programming model for two- and three-stage two-dimensional
  cutting stock problems.
\newblock \emph{European Journal of Operational Research}, 205\penalty0
  (3):\penalty0 699--708, 2010.

\bibitem[Trick(2003)]{T03}
M.A. Trick.
\newblock A dynamic programming approach for consistency and propagation for
  knapsack constraints.
\newblock \emph{Annals of Operations Research}, 118\penalty0 (1-4):\penalty0
  73--84, 2003.

\bibitem[Uchoa(2012)]{U11}
E.~Uchoa.
\newblock Cuts over extended formulations by flow discretization.
\newblock In A.R. Mahjoub, editor, \emph{Progress in Combinatorial
  Optimization}, chapter~8, pages 255--282. Wiley, 2012.

\bibitem[{Val{\'e}rio de Carvalho}(1999)]{V99}
J.M. {Val{\'e}rio de Carvalho}.
\newblock Exact solution of bin-packing problems using column generation and
  branch-and-bound.
\newblock \emph{Annals of Operations Research}, 86:\penalty0 629--659, 1999.

\bibitem[Vanderbeck(2000)]{V00}
F.~Vanderbeck.
\newblock On dantzig-wolfe decomposition in integer programming and ways to
  perform branching in a branch-and-price algorithm.
\newblock \emph{Operations Research}, 48\penalty0 (1):\penalty0 111--128, 2000.

\bibitem[Vanderbeck(2011)]{V10}
F.~Vanderbeck.
\newblock Branching in branch-and-price: a generic scheme.
\newblock \emph{Mathematical Programming}, 130:\penalty0 249--294, 2011.

\bibitem[Vu et~al.(2020)Vu, Hewitt, Boland, and Savelsbergh]{VHBS19}
D.M. Vu, M.~Hewitt, N.~Boland, and M.~Savelsbergh.
\newblock Dynamic discretization discovery for solving the time-dependent
  traveling salesman problem with time windows.
\newblock \emph{Transportation Science}, 54\penalty0 (3):\penalty0 703--720,
  2020.

\bibitem[Wei et~al.(2020)Wei, Luo, Baldacci, and Lim]{WLBL20}
L.~Wei, Z.~Luo, R.~Baldacci, and A.~Lim.
\newblock A new branch-and-price-and-cut algorithm for one-dimensional
  bin-packing problems.
\newblock \emph{INFORMS Journal on Computing}, 32\penalty0 (2):\penalty0
  428--443, 2020.

\end{thebibliography}

\end{document}